\documentclass[twoside,11pt,reqno]{amsart}
\usepackage{amsmath,amssymb,amscd,mathrsfs,epic,empheq}
\usepackage{latexsym,amsthm,amscd, enumerate}

\makeatletter

\newif\ifbrauer@
\brauer@true

\@addtoreset{equation}{section}
\@addtoreset{figure}{section}

\newtheorem{Proposition}{Proposition}[section]

\newtheorem{Theorem}[Proposition]{Theorem}

\def\Ustacked#1{{\:\stackrel{#1}{_{\tU}}\begin{picture}(0,0) \put(-5.6,4.5){\oval(9,15)[t]}\put(-5.6,4.5){\oval(9,15)[b]}\end{picture}}}
\def\Vstacked#1{{\:\stackrel{#1}{_{\tV}}\begin{picture}(0,0) \put(-5.6,4.5){\oval(9,15)[t]}\put(-5.6,4.5){\oval(9,15)[b]}\end{picture}}}
\def\Tstacked#1{{\:\stackrel{#1}{_{\T}}\begin{picture}(0,0) \put(-5.6,4.5){\oval(9,15)[t]}\put(-5.6,4.5){\oval(9,15)[b]}\end{picture}}}
\def\Sstacked#1{{\:\stackrel{#1}{_{\Stab}}\begin{picture}(0,0) \put(-5.6,4.5){\oval(9,15)[t]}\put(-5.6,4.5){\oval(9,15)[b]}\end{picture}}}
\def\Astacked#1{{\:\stackrel{#1}{_{\tA}}\begin{picture}(0,0) \put(-5.6,4.5){\oval(9,15)[t]}\put(-5.6,4.5){\oval(9,15)[b]}\end{picture}}}
\def\Cstacked#1{{\:\stackrel{#1}{_{\tC}}\begin{picture}(0,0) \put(-5.6,4.5){\oval(9,15)[t]}\put(-5.6,4.5){\oval(9,15)[b]}\end{picture}}}

\newbox\squ  
\setbox\squ=\hbox{\vrule width.6pt
   \vbox{\hrule height.6pt width.4em\kern1ex
         \hrule height.6pt}%
   \vrule width.6pt}

\def\0{{\bar 0}}
\def\1{{\bar 1}}
\def\Tab{\mathscr{T}}
\def\Hfin{H}
\def\Haff{\widehat{H}}

\def\tA{{\mathtt A}}
\def\tB{{\mathtt B}}
\def\tC{{\mathtt C}}
\def\tD{{\mathtt D}}
\def\T{{\mathtt T}}
\def\tU{{\mathtt U}}
\def\tV{{\mathtt V}}
\def\Stab{{\mathtt S}}

\def\cO{\mathcal{O}}

\def\ev{\operatorname{ev}}

\def\deg{\operatorname{deg}}
\def\defect{\operatorname{def}}

\def\wt{{\operatorname{wt}}}

\def\op{\operatorname{op}}

\def\down{{\scriptstyle\vee}}
\def\up{{\scriptstyle\wedge}}
\def\cross{{\scriptstyle\times}}

\def\Rep#1{\operatorname{Rep}(#1)}

\def\ev{{\operatorname{ev}}}

\def\C{{\mathbb C}}

\def\Z{{\mathbb Z}}

\def\hom{{\operatorname{Hom}}}

\def\End{{\operatorname{End}}}

\def\bz{\hbox{\boldmath{$0$}}}
\def\eps{{\varepsilon}}

\def\phi{{\varphi}}

\def\la{{\lambda}}
\def\La{{\Lambda}}

\def\@underbar#1{\settowidth{\@tempdimb}{#1}\@tempdimb=0.8\@tempdimb
                   \ooalign{#1\crcr
                         \hfil\rule[-.5mm]{\@tempdimb}{.4pt}\hfil}}

\def\bi{\text{\boldmath$i$}}

\newdimen\hoogte    \hoogte=14pt    
\newdimen\breedte   \breedte=14pt   
\newdimen\dikte     \dikte=0.5pt    

\newenvironment{young}{\begingroup
       \def\vr{\vrule height0.8\hoogte width\dikte depth 0.2\hoogte}
       \def\fbox##1{\vbox{\offinterlineskip
                    \hrule height\dikte
                    \hbox to \breedte{\vr\hfill##1\hfill\vr}
                    \hrule height\dikte}}
       \vbox\bgroup \offinterlineskip \tabskip=-\dikte \lineskip=-\dikte
            \halign\bgroup &\fbox{##\unskip}\unskip  \crcr }
       {\egroup\egroup\endgroup}
\def\diagram#1{\relax\ifmmode\vcenter{\,\begin{young}#1\end{young}\,}\else%
              $\vcenter{\,\begin{young}#1\end{young}\,}$\fi}

\begin{document}

\title[Level two Hecke algebras]{\boldmath 
An orthogonal form for level two Hecke algebras with applications
}

\author{Jonathan Brundan}

\address{Department of Mathematics, University of Oregon, Eugene, OR 97403, USA}
\email{brundan@uoregon.edu}

\thanks{Research supported in part by NSF grant no. DMS-0654147}

\begin{abstract}
\ifbrauer@
This is a survey of some recent results relating Khovanov's arc algebra to 
category $\mathcal O$ for Grassmannians, the general linear supergroup, 
and
the walled Brauer algebra.
\fi
The exposition emphasizes 
an extension of Young's orthogonal form for 
level two cyclotomic Hecke algebras.
\end{abstract}
\maketitle

\section{Introduction}

\ifbrauer@
This article is primarily intended as a survey of some of my recent joint results with Catharina Stroppel from \cite{BS1}--\cite{BSnew}. 
In that work, we 
exploited the isomorphism constructed in \cite{BKyoung}
between level two cyclotomic quotients of certain affine Hecke algebras
and quiver Hecke algebras
to establish some remarkable 
connections between Khovanov's arc algebra from \cite{K2},
the Bernstein-Gelfand-Gelfand
category $\mathcal O$ for Grassmannians,
the general linear supergroup $GL(m|n)$, and the walled Brauer algebra $B_{r,s}(\delta)$.
\fi
Our results can be viewed as an application of some of the
ideas emerging from the development of
higher representation theory by Khovanov and Lauda \cite{KLa} and 
Rouquier \cite{Rou}; see the recent survey \cite{SICM} 
which adopts this point of view.
We are going to approach the subject instead from a more classical direction, focussing
primarily on a certain extension of Young's orthogonal form for 
level two Hecke algebras. This orthogonal form is really the technical heart of the paper \cite{BS3} but it is quite well hidden. We will also explain its $q$-analogue not mentioned at all there.

The level two Hecke algebras studied here include as a special case the
usual finite Iwahori-Hecke algebras of type $B$ when the long root parameter $q$ is generic and the short root parameter $Q$ is chosen so that the algebra is {\em not} semisimple (so $Q = -q^r$ for some $r$).
The orthogonal form allows
many of usual problems of
representation theory
to be solved for these algebras in an unusually explicit way. For example it yields
constructions of all the irreducible modules, hence we can compute their dimensions, and all the
projective indecomposable modules, hence we can identify the endomorphism algebra of a minimal projective generator.
It is a tantalizing problem to try to find something like this
for cyclotomic Hecke algebras of higher
levels, or at roots of unity, but at the moment this seems out of reach.

Once we have explained the orthogonal form, we discuss the main applications 
obtained in 
\ifbrauer@
\cite{BS3}--\cite{BSnew}.
These applications rely also on three generalizations of Schur-Weyl duality.
The first of these generalizations was developed in detail already in
\cite{BKschur}, \cite{BKariki}, and is exploited in \cite{BS3} 
to relate the level two Hecke algebras 
to category $\mathcal O$ for Grassmannians.
The second generalized Schur-Weyl duality appears for the first time in \cite{BS4},
and relates the same Hecke algebras to finite dimensional representations of the complex general linear supergroup.
Finally in \cite{BSnew} we use 
the Schur-Weyl duality between $GL(m|n)$ and the walled Brauer algebra arising from ``mixed'' tensor space
to prove a conjecture suggested by Cox and De Visscher \cite{CD}.
\fi

In the remainder of the article,
in an attempt to improve readability, we have postponed precise references to notes at the end of each section.
We fix once and for all a ground field $F$ and a parameter $\xi \in F^\times$ such that
{\em either}\,
 $\xi$ is not a root of unity in $F$,
{\em or}\,
$\xi=1$ and $F$ is of characteristic zero.
For the applications beginning in section 5, we always take $F = \C$ and $\xi=1$.

\vspace{2mm}
\noindent
{\em Acknowledgements.} I am grateful to Toshiaki Shoji and all 
the other organizers of the conference
``Representation Theory of Algebraic Groups and Quantum Groups" in Nagoya in August 2010 for giving me 
the opportunity both to speak and to write on this topic.
Thanks also to Catharina Stroppel for comments on the first draft.

\section{Level one Hecke algebras}

We begin by recalling briefly 
Young's classical orthogonal form for the symmetric group
and its (not quite so classical) analogue for the corresponding
Iwahori-Hecke algebra.
Let $\Hfin_d$ denote the finite Iwahori-Hecke algebra associated to the symmetric group $S_d$
over the field $F$ at defining parameter $\xi$.
Thus $\Hfin_d$ is a finite dimensional algebra of dimension $d!$,
with 
generators $T_1,\dots,T_{d-1}$ subject to the usual braid relations plus the quadratic relations
\begin{equation}\label{qrel}
(T_r+1)(T_r-\xi) = 0
\end{equation}
for each $r=1,\dots,d-1$.
In the {degenerate case} $\xi =1$ we can simply identify $\Hfin_d$ with the group algebra $F S_d$ of the symmetric group, so that $T_r$ is identified with the basic transposition $s_r := (r\:\:r\!+\!1)$.
The assumptions on $F$ and $\xi$
mean that $\Hfin_d$ is a semisimple
algebra.
Up to isomorphism, the irreducible $\Hfin_d$-modules are the 
{\em Specht modules} 
$\{S(\lambda)\:|\:\lambda \vdash d\}$ 
parametrized by partitions of $d$.

Given a partition $\lambda \vdash d$, 
we draw the Young diagram of $\la$ in the usual English way. A 
{\em $\la$-tableau} means a filling of the boxes of this
Young diagram
with the entries $1,\dots,d$ (each appearing exactly once).
The symmetric group $S_d$ acts on such tableaux via its natural action
on the entries.
Let $\Tab(\la)$ denote the set of all {\em standard} $\la$-tableaux,
that is, the ones whose entries
are strictly increasing both along rows from left to right and down
columns from top to bottom.
The {\em residue sequence} $\bi^\T \in \Z^d$ of $\T \in \Tab(\la)$
is the sequence $(i_1,\dots,i_d)$ where $i_r$
is the {\em residue} of the box of $\T$ containing entry $r$, that is, the
integer $(c-b)$ if this box is in row $b$ and column $c$.
Of course residues are constant along diagonals; they are the numbers labelling the boundary of $
\T$ in the following example:

$$
\phantom{\T = \qquad\diagram{1&3&4&9\cr2&5&7\cr6&8&10\cr}}
\qquad\leftrightarrow\qquad
\bi^\T = (0,-1,1,2,0,-2,1,-1,3,0)\hspace{-1.3mm}
\begin{picture}(0,0)
\put(-315.1,0){$\T = \qquad\diagram{1&3&4&9\cr2&5&7\cr6&8&10\cr}$}
\put(-282.7,-11){$\!\!\!\!_{-3}$}\put(-279,-16.8){$_{\smallsetminus}$}
\put(-282.7,3.5){$\!\!\!\!_{-2}$}\put(-279,-2.3){$_{\smallsetminus}$}
\put(-282.2,17){$\!\!\!\!_{-1}$}\put(-279,12.2){$_{\smallsetminus}$}
\put(-282.9,30.5){$_{0}$}\put(-279,26.7){$_{\smallsetminus}$}
\put(-269.4,30.5){$_{1}$}\put(-265.5,26.7){$_{\smallsetminus}$}
\put(-255.9,30.5){$_{2}$}\put(-252,26.7){$_{\smallsetminus}$}
\put(-242.4,30.5){$_{3}$}\put(-238.5,26.7){$_{\smallsetminus}$}
\put(-228.9,30.5){$_{4}$}\put(-225,26.7){$_{\smallsetminus}$}
\end{picture}
$$
Notice the residue sequence of a standard tableau always has the property that $|i_r - i_{r+1}|
\geq 1$ for every $r$, and $|i_r-i_{r+1}| > 1$ if and only if the tableau 
$s_r \cdot \T$ is again standard. Moreover the original standard tableau $\T$
can be recovered uniquely from the sequence $\bi^\T$, because for any
Young diagram there is at most one ``addable'' box of a given residue.

Young's orthogonal form gives an explicit construction
of each $S(\lambda)$ as a vector space with a
distinguished basis on
which the actions of the generators of $\Hfin_d$ are given by
explicit formulae. To write it down, let
$$
[n] := \left\{\begin{array}{ll}
n&\text{if $\xi=1$,}\\
\frac{\xi^n-1}{\xi-1}&\text{if $\xi \neq 1$,}
\end{array}\right.
$$
for any $n\in \Z$.

\begin{Theorem}[Young's orthogonal form]\label{yot}
For $\lambda \vdash d$, the irreducible $\Hfin_d$-module 
$S(\lambda)$ has a basis
$\{v_\T\:|\:\T \in \Tab(\lambda)\}$ on which $T_r \in \Hfin_d$ 
acts by
\begin{equation}\label{ysf}
T_r v_\T :=  \left(\psi_r \left(1-\textstyle{\frac{1}{[i_r-i_{r+1}]}}\right) - {\textstyle\frac{1}{[i_r-i_{r+1}]}}\right) v_\T,
\end{equation}
where $\bi := \bi^\T$ and
$\psi_r$ is the endomorphism with
\begin{equation}\label{gym}
\psi_r v_{\T}
:=
\left\{
\begin{array}{ll}
v_{s_r\cdot \T} &\hbox{if $s_r\cdot\T\in \Tab(\la)$},\\
0 &\hbox{otherwise.}
\end{array}
\right.
\end{equation}
\end{Theorem}

Recall also that the {\em Jucys-Murphy elements}
in $\Hfin_d$ are the commuting elements
$1 = L_1,\dots,L_d$ defined from
$L_r := \xi^{1-r} T_{r-1} \cdots T_2 T_1 T_2 \cdots T_{r-1}$
in the case $\xi \neq 1$, or the elements
$0 = L_1,\dots,L_d$ defined from
$L_r := \sum_{1 \leq s < r} (s\:r)$ in the case $\xi = 1$.
Although not obvious from (\ref{ysf}), 
the Jucys-Murphy elements act on Young's basis 
so that 
\begin{equation}\label{wb}
L_r v_\T =
\left\{
\begin{array}{ll}
i_r v_\T&\text{if $\xi = 1$,}\\
\xi^{i_r} v_\T\qquad\qquad&\text{if $\xi \neq 1$,}
\end{array}\right.
\end{equation}
where $\bi^\T = (i_1,\dots,i_d)$.
Thus Young's basis consists of simultaneous eigenvectors for the Jucys-Murphy elements.

We end the section by defining some
explicit but rather complicated power series $p_r(\bi), q_r(\bi) \in F[[y_1,\dots,y_d]]$.
These will be needed at a crucial point in the next section 
in order to write down our extension of Young's orthogonal form.
For fixed $\bi \in \Z^d$ and $1 \leq r < d$, set
\begin{align}
N &:= \left\{
\begin{array}{ll}
(y_{r+1}+i_{r+1}+1)-(y_r+i_r)&\text{if $\xi = 1$,}\\
\xi^{i_r}(1-y_r)-\xi^{i_{r+1}+1}(1-y_{r+1})&\text{if $\xi \neq 1$,}
\end{array}
\right.\\
D &:= \left\{
\begin{array}{ll}
(y_{r+1}+i_{r+1})-(y_r+i_r)\hspace{10.4mm}&\text{if $\xi = 1$,}\\
\xi^{i_r}(1-y_r)-\xi^{i_{r+1}}(1-y_{r+1})&\text{if $\xi \neq 1$.}
\end{array}
\right.
\end{align}
Note $D$ is a unit in $F[[y_1,\dots,y_d]]$ if $i_r \neq i_{r+1}$, so it then makes sense to define
\begin{align}\label{pdef}
p_r(\bi) &:= \left\{
\begin{array}{ll}
1&\text{if $i_r = i_{r+1}$,}\\
1-N/D&\text{if $i_r \neq i_{r+1}$},
\end{array}
\right.\\
q_r(\bi) &:= \left\{
\begin{array}{ll}
\xi^{-i_r} N\!\hspace{4.9mm}&\text{if $i_r = i_{r+1}$,}\\
N/D&\text{if $|i_r-i_{r+1}|>1$,}\\
N/D^2&\text{if $i_r = i_{r+1}-1$,}\\
\xi^{i_r}&\text{if $i_r = i_{r+1}+1$.}
\end{array}
\right.\label{qdef}
\end{align}
To make the connection with Theorem~\ref{yot}, we 
observe 
on setting $y_1=\cdots=y_d=0$ that
$p_r(\bi)$ evaluates to
${\textstyle\frac{1}{[i_r-i_{r+1}]}}$ 
assuming $i_r \neq i_{r+1}$, and 
$q_r(\bi)$ evaluates to $1-\textstyle{\frac{1}{[i_r-i_{r+1}]}}$
assuming
$|i_r-i_{r+1}|> 1$.
So 
taking $y_1=\cdots=y_d=0$ 
the formula (\ref{ysf}) can be rewritten as
\begin{equation}\label{ysfcomp}
T_r v_{\T} = (\psi_r q_r(\bi) - p_r(\bi)) v_{\T},
\end{equation}
for any $\T \in \Tab(\lambda)$ and  $\bi := \bi^\T$.

\vspace{2mm}
\noindent
{\em Notes.}
Theorem~\ref{yot} originates in \cite{Young}, and its extension to the Iwahori-Hecke algebra was worked out by
Hoefsmit in \cite{H}.
The account here is based closely on \cite[$\S$5]{BKyoung} in which
the endomorphisms $\psi_r$ and $y_r$ are interpreted
as certain {\em Khovanov-Lauda-Rouquier generators} for $\Hfin_d$,
satisfying the defining relations of the cyclotomic quiver Hecke algebras of
\cite{KLa}, \cite{Rou} attached to the infinite linear quiver $A_\infty$ and the fundamental
dominant weight $\Lambda_0$.
The formulae (\ref{pdef}), (\ref{qdef}) are exactly
\cite[(3.22), (3.30)]{BKyoung} if $\xi = 1$
and \cite[(4.27), (4.36)]{BKyoung} if $\xi \neq 1$.

\section{Level two Hecke algebras}

Continuing to work over the ground field $F$, let $\Haff_d$ be the affine Hecke algebra 
on generators
$\{X_1^{\pm 1},\dots,X_d^{\pm 1}\} \cup \{T_1,\dots,T_{d-1}\}$
if $\xi \neq 1$, or its degenerate analogue on generators
$\{x_1,\dots,x_d\}\cup\{s_1,\dots,s_{d-1}\}$
if $\xi = 1$; in the latter case it is convenient to set
$X_r := x_r$ and $T_r:=s_r$. 
The relations are as follows:
the $X_r$'s commute,
the $T_r$'s satisfy the defining relations of
the finite Iwahori-Hecke algebra $\Hfin_d$ from (\ref{qrel}),
$T_r X_{s} = X_{s} T_r$ if $s \neq r,r+1$,
and finally 
\begin{equation*}
\left\{
\begin{array}{rll}
&T_r X_r T_r = \xi X_{r+1}
&\qquad\text{if $\xi \neq 1$,}\\
&\,\,s_r x_{r+1} = x_r s_r+1
&\qquad\text{if $\xi = 1$.}
\end{array}\right.
\end{equation*}
The finite Iwahori-Hecke algebra $\Hfin_d$ is a subalgebra of $\Haff_d$
in the obvious way. Moreover it is also a quotient algebra in many
different ways: for each $r \in \Z$
there is an {\em evaluation homomorphism}
\begin{equation}\label{Eval}
\ev_r:\Haff_d \twoheadrightarrow \Hfin_d
\end{equation}
which is the identity on the subalgebra $\Hfin_d$
and maps $X_1 \mapsto \xi^r$ if $\xi \neq 1$ or $x_1 \mapsto r$ if
$\xi = 1$.
The Jucys-Murphy element $L_r\in\Hfin_d$
from (\ref{wb}) is $\ev_0(X_r)$.

More generally, we can consider the quotient of $\Haff_d$ by the
two-sided ideal generated by a monic
polynomial of degree $k$ in $X_1$.
This gives a finite dimensional algebra of
dimension $k^d d!$ known as an {\em Ariki-Koike algebra} of level $k$ or a 
{\em cyclotomic Hecke algebra} of type $G(k,1,d)$.
The original finite Iwahori-Hecke algebra $\Hfin_d$ corresponds to level one.
In the remainder of the article we are interested in the level two
case.
So we fix integers $p, q\in \Z$ and set
\begin{equation}\label{ker}
H_d^{p,q} := \left\{
\begin{array}{ll}
\Haff_d \Big/ \big\langle \,(X_1-\xi^p)(X_1-\xi^q)\,\big\rangle
&\text{if $\xi \neq 1$,}\\
\Haff_d \Big/ \big\langle \,(x_1-p)(x_1-q)\,\big\rangle
&\text{if $\xi = 1$},
\end{array}\right.
\end{equation}
of dimension $2^d d!$.
We will use the same notation for the generators $T_1,\dots,T_{d-1}$ 
of $\Haff_d$ and for their canonical images in $H_d^{p,q}$,
and denote the canonical images of $X_1,\dots,X_{d}$ by $L_1,\dots,L_{d}$.

The algebra $H_d^{p,q}$ is semisimple if and only if $d \leq |q-p|$,
in which case its representation theory is just as easy as the level
one case discussed in the previous section.
It turns out that the representation theory of $H_d^{p,q}$ is still 
very manageable even when it is not semisimple. In fact it provides
delightful ``baby model'' for the representation theory of arbitrary
cyclotomic Hecke algebras.
We still have {\em Specht modules} $S(\lambda)$ but they are no longer irreducible;
they are parametrized now by
{\em bipartitions} $\lambda \Vdash d$, which are ordered 
pairs $\lambda = (\lambda^L, \lambda^R)$ of partitions
$\lambda^L \vdash a$ and $\lambda^R 
\vdash b$ such that $d=a+b$. For such a bipartition $\lambda$, 
the corresponding Specht module is
\begin{equation}
S(\lambda) := \Haff_d \otimes_{\Haff_{a} \otimes \Haff_{b}} \left( 
\ev_p^* S(\lambda^L) \boxtimes \ev_q^* S(\lambda^R)
\right),
\end{equation}
where $\Haff_{a} \otimes \Haff_{b}$ is the parabolic
subalgebra of $\Haff_d$, and 
$\ev_p^* S(\lambda^L) \boxtimes \ev_q^* S(\lambda^R)$ denotes the
$\Haff_{a} \otimes \Haff_{b}$-module arising as the outer tensor
product
of level one Specht modules $S(\lambda^L)$ and $S(\lambda^R)$
viewed as modules over $\Haff_{a}$ and $\Haff_{b}$, respectively, via the
evaluation homomorphisms $\ev_p:\Haff_{a} \twoheadrightarrow
\Hfin_{a}$ 
and $\ev_q:\Haff_{b} \twoheadrightarrow \Hfin_{b}$ as in (\ref{Eval}).
This induced module is {\em a priori} an $\Haff_d$-module, but one can
check from (\ref{ker}) that it factors through to the
quotient $H_d^{p,q}$, hence $S(\lambda)$ is a well-defined
$H_d^{p,q}$-module.

We say that a bipartition $\lambda \Vdash d$
is {\em restricted} if the appropriate one of the following conditions holds
for each $i \geq 1$:
\begin{equation}\label{resdef}
\left\{
\begin{array}{ll}
\lambda_i^L \leq \lambda_i^R+q-p\qquad&\text{if $p \leq q$,}\\
\lambda_{i+p-q}^L \leq \lambda_{i}^R&\text{if $p \geq q$,}
\end{array}
\right.
\end{equation} 
where $\lambda_1^L \geq \lambda_2^L \geq \cdots$ are the parts of $\lambda^L$ and 
$\lambda_1^R \geq \lambda_2^R \geq \cdots$ are the parts of $\lambda^R$. 
The following theorem gives a classification of the irreducible
$H_d^{p,q}$-modules.

\begin{Theorem}\label{vazt}
If $\lambda \Vdash d$ is restricted, then the Specht module
$S(\lambda)$ has a unique irreducible quotient 
denoted $D(\lambda)$, and the modules
\begin{equation}\label{ic}
\{D(\lambda)\:|\:\text{for all restricted }\lambda \Vdash d\}
\end{equation}
give a complete set of pairwise inequivalent irreducible $H_d^{p,q}$-modules.
\end{Theorem}

Next we want to describe the
{\em composition multiplicities} $[S(\lambda):D(\mu)]$ of
Specht modules, all of which turn out to be either zero or one. 
First we need some combinatorics.
By a {\em weight diagram} we mean 
a horizontal number line with vertices 
at all integers labelled by 
one of the symbols $\circ, \down, \up$ and $\cross$; we require
moreover that it is impossible to find a vertex labelled $\down$ to the left of a vertex labelled $\up$ outside of some finite subset of the vertices.
We always identify bipartitions with 
particular weight diagrams so that $\lambda \Vdash d$ corresponds to the weight
diagram obtained by putting the symbol $\down$ at all the vertices 
indexed by the set
$\{p+\lambda_1^L, p+\lambda_2^L-1, p+\lambda_3^L-2,\dots\}$,
the symbol $\up$ at all the vertices indexed by the set
$\{q+\lambda_1^R, q+\lambda_2^R-1, q+\lambda_3^R-2,\dots\}$,
and interpreting vertices labelled
both $\down$ and $\up$ as the label $\cross$
and vertices labelled neither $\down$ nor $\up$ as the label $\circ$.
Of course this depends implicitly on the fixed choices of $p$ and $q$.
Here are some examples:
\begin{align*}
(\varnothing, \varnothing)
&=
\!\!\!\begin{picture}(240,12)
\put(90,14){$_p$}
\put(8,-.3){$\cdots$}
\put(227,-.3){$\cdots$}
\put(25,2.3){\line(1,0){198}}
\put(30,.4){$\cross$}
\put(50,.4){$\cross$}
\put(70,.4){$\cross$}
\put(90,.4){$\cross$}
\put(110,-2.4){$\up$}
\put(130,-2.4){$\up$}
\put(150,-2.4){$\up$}
\put(170,-.4){$\circ$}
\put(190,-.4){$\circ$}
\put(210,-.4){$\circ$}
\end{picture}\,,&q&=p+3,
\\  ((1), (3^2 2))&=
\!\!\!\begin{picture}(240,0)
\put(8,-.3){$\cdots$}
\put(227,-.3){$\cdots$}
\put(25,2.3){\line(1,0){198}}
\put(30,.4){$\cross$}
\put(50,2.4){$\down$}
\put(70,2.4){$\down$}
\put(90,-2.4){$\up$}
\put(110,2.4){$\down$}
\put(130,-2.4){$\up$}
\put(150,-2.4){$\up$}
\put(170,-.4){$\circ$}
\put(190,-.4){$\circ$}
\put(210,-.4){$\circ$}
\end{picture}\,,
&q&=p,
\\ ((5 3^2), (41))
&=
\!\!\!\begin{picture}(240,0)
\put(8,-.3){$\cdots$}
\put(227,-.3){$\cdots$}
\put(25,2.3){\line(1,0){198}}
\put(30,.4){$\cross$}
\put(50,-.4){$\circ$}
\put(70,-2.4){$\up$}
\put(90,-.4){$\circ$}
\put(110,2.4){$\down$}
\put(130,2.4){$\down$}
\put(150,-2.4){$\up$}
\put(170,-.4){$\circ$}
\put(190,2.4){$\down$}
\put(210,-.4){$\circ$}
\end{picture}\,,&q&=p-1.
\end{align*}
These examples
have infinitely many vertices labelled
$\cross$ to the left and infinitely many vertices labelled $\circ$ to the right,
as do {\em all} weight diagrams arising from bipartitions.
Later in the article, we will meet other sorts of weight diagrams which 
are not of this form.

Given a weight diagram $\lambda$, a {\em $\lambda$-cap diagram} is a diagram obtained by
attaching 
caps $\cap$ and rays up to infinity $|$
to
all the vertices of $\lambda$ labelled $\down$ or $\up$, so that there are no crossings of caps and/or rays,
the labels at the ends of caps are either 
$\begin{picture}(18,11) 
\put(10,1){\oval(14,20)[t]}
\put(0.2,.8){$\down$}
\put(14.2,-3){$\up$}
\end{picture}$ (``counter-clockwise'')
or 
$\begin{picture}(18,11) 
\put(10,1){\oval(14,20)[t]}
\put(14.2,.8){$\down$}
\put(.2,-3){$\up$}
\end{picture}$ (``clockwise''), all rays labelled $\up$ are strictly
to the left of all rays labelled $\down$, and the total number of caps is finite.
Here are some examples:
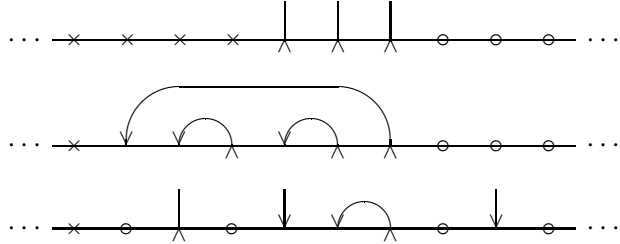
\begin{align*}
&\begin{picture}(240,27)
\put(8,12.7){$\cdots$}
\put(227,12.7){$\cdots$}
\put(25,15.4){\line(1,0){198}}
\put(30,13.4){$\cross$}
\put(50,13.4){$\cross$}
\put(70,13.4){$\cross$}
\put(90,13.4){$\cross$}
\put(110.1,10.7){$\up$}
\put(130.1,10.7){$\up$}
\put(150.1,10.7){$\up$}
\put(170,12.6){$\circ$}
\put(190,12.6){$\circ$}
\put(210,12.6){$\circ$}
\put(112.9,15){\line(0,1){15}}
\put(132.9,15){\line(0,1){15}}
\put(152.9,15){\line(0,1){15}}
\end{picture}
\\
&\begin{picture}(240,27)
\put(122.9,11.2){\oval(20,20)[t]}
\put(82.9,11.2){\oval(20,20)[t]}
\put(102.9,11.2){\oval(100,44)[t]}
\put(8,8.4){$\cdots$}
\put(227,8.4){$\cdots$}
\put(25,11){\line(1,0){198}}
\put(30,9.1){$\cross$}
\put(50,11.5){$\down$}
\put(70,11.5){$\down$}
\put(90,6.5){$\up$}
\put(110,11.5){$\down$}
\put(130,6.5){$\up$}
\put(150,6.5){$\up$}
\put(170,8.3){$\circ$}
\put(190,8.3){$\circ$}
\put(210,8.3){$\circ$}
\end{picture}
\\
&\begin{picture}(240,17)
\put(8,2.7){$\cdots$}
\put(227,2.7){$\cdots$}
\put(25,5.3){\line(1,0){198}}
\put(30,3.4){$\cross$}
\put(50,2.6){$\circ$}
\put(70,.7){$\up$}
\put(90,2.6){$\circ$}
\put(110,5.7){$\down$}
\put(130,5.7){$\down$}
\put(150,.7){$\up$}
\put(170,2.6){$\circ$}
\put(190,5.7){$\down$}
\put(210,2.6){$\circ$}
\put(143,5.4){\oval(20,20)[t]}
\put(192.8,5){\line(0,1){15}}
\put(112.8,5){\line(0,1){15}}
\put(72.8,5){\line(0,1){15}}
\end{picture}
\end{align*}
\noindent
The {\em weight} $\wt(\tA)$ of a $\lambda$-cap diagram $\tA$
is the weight diagram obtained from $\la$
by switching the labels at the
ends of all the clockwise caps of $\tA$. 
Observe in particular that there is always a unique $\lambda$-cap diagram of weight $\lambda$.

We say that a $\lambda$-cap diagram is {\em restricted}
if all its rays are labelled in the same way; in the above examples, the first two are restricted but the third is not.
We say that the weight diagram $\lambda$ itself is
{restricted} if the unique $\lambda$-cap diagram of weight $\lambda$ is restricted. 
In the case that $\lambda$ is the weight diagram arising from a bipartition
of $d$, $\lambda$ is restricted as a weight diagram 
if and only if it is a restricted bipartition
in the sense of (\ref{resdef}); e.g. in our running example the first two bipartitions are restricted, but the third is not.

Finally define a reflexive and anti-symmetric 
relation $\supset$ on weight diagrams by declaring that
$\lambda \supset \mu$ 
if there exists a (necessarily unique) $\la$-cap diagram of weight $\mu$. 
For fixed $\mu$ it is easy to find all $\la$ such that 
$\la \supset \mu$: they are all the weight diagrams that can be 
obtained from $\mu$ 
by switching the labels at the ends of some subset 
of the caps in the unique $\mu$-cap diagram of weight $\mu$.
It follows easily that the transitive closure of the relation $\supset$
is the same as the {\em Bruhat order} $\geq$ on weight diagrams generated by 
the elementary relation $\cdots\up\cdots\down\cdots \geq\cdots\down\cdots\up\cdots$.
On the other hand, for fixed $\lambda$ it is trickier to find all $\mu$ such that $\lambda \supset \mu$.
For example if $\lambda = 
\up\:\down\:\up\:\down$ (and all other vertices
are labelled $\circ$ or $\times$) there are five $\lambda$-cap diagrams
hence five weights $\mu$ with $\lambda \supset \mu$,
namely,
$\up\:\down\:\up\:\down$,
$\down\:\up\:\up\:\down$,
$\up\:\down\:\down\:\up$,
$\down\:\down\:\up\:\up$ and
$\down\:\up\:\down\:\up$. It is no coincidence here that 
the third Catalan number $C_3 = 5$: consider
$\lambda = \up\:\down\:\up\:\down\:\up\:\down$ to get $C_4$ and so on.

\begin{Theorem}\label{cm}
For $\lambda, \mu \Vdash d$ with $\mu$ restricted, we have that
$$
[S(\lambda):D(\mu)] = \left\{
\begin{array}{ll}
1&\text{if $\lambda \supset \mu$,}\\
0&\text{otherwise.}
\end{array}
\right.
$$
\end{Theorem}

It is already clear from this that we are in a rather unusual
situation. 
In fact, much more is possible: there is a remarkable extension of
Young's orthogonal form for the algebra $H_d^{p,q}$
giving an explicit construction of
another family of $H_d^{p,q}$-modules
denoted $\{Y(\lambda)\:|\:\lambda \Vdash d\}$. 
As usual, we need some more combinatorial preparation.
For $\lambda \Vdash d$, the {\em Young diagram} of $\lambda$
means the ordered pair of the 
Young diagrams of $\lambda^L$ and $\lambda^R$. 
A {\em $\lambda$-tableau}
$\T = (\T^L, \T^R)$ means a filling of the boxes of this diagram by the numbers $1,\dots,d$ (each appearing
exactly once), and as in the previous section the symmetric group $S_d$ acts
on $\lambda$-tableaux by its action on the entries. 
We let $\Tab(\lambda)$ denote the set of all 
{\em standard} $\lambda$-tableaux, that is, the $\T = (\T^L,\T^R)$ 
such that the entries of both $\T^L$ and $\T^R$ 
increase strictly along rows and down columns.
The {\em residue} of the box in the $b$th row
and $c$th column of the Young diagram of $\lambda^L$ 
(resp.\ $\lambda^R$) is $p + b-c$ (resp. $q+b-c$).
Then the {\em residue sequence} $\bi^\T \in \Z^d$ of $\T \in \Tab(\lambda)$
is defined just like in the previous section.
For example:
$$
\phantom{\T = \Bigg(\quad\:\:
\diagram{2&5&6\cr 3&8\cr}\:\:
,\quad 
\diagram{1&4\cr 7 \cr}
\:\:
\Bigg)}
\begin{picture}(0,26)
\put(-145.9,0.1){$\T = \Bigg(\quad\:\:
\diagram{2&5&6\cr 3&8\cr}\:\:
,\quad 
\diagram{1&4\cr 7 \cr}
\:\:
\Bigg)$}
\put(-111.2,24){$_{p}$}\put(-106.6,19){$_{\smallsetminus}$}
\put(-50.7,24){$_{q}$}\put(-45.9,19){$_{\smallsetminus}$}
\end{picture}
\quad
\leftrightarrow\quad
\bi^\T = (q,p,p-1,q+1,p+1,p+2,q-1,p)
$$
Now for $\lambda \Vdash d$, 
we can define the {\em Young module}
$Y(\lambda)$
to be the vector space on basis
$\big\{v_\T^\lambda\:\big|\:\
\T \in \bigcup_{\mu \supset \lambda} \Tab(\mu)
\big\}.$ To define the action of $H_d^{p,q}$,
take $\T \in \Tab(\mu)$ for some $\mu \supset \lambda$ and
set $\bi := \bi^\T$.
At the end of the section we will define endomorphisms
$y_r$ and $\psi_r$ of the vector space $Y(\lambda)$
such that
\begin{align}\label{yrprop}
y_r v^\lambda_\T &\in \Big\langle \:v^\lambda_\Stab\:\:\Big|\:
\Stab \in {\textstyle\bigcup_{\mu \leq \nu \supset\lambda}}\Tab(\nu), \bi^\Stab =
\bi \Big\rangle,\\
\psi_r v^\lambda_\T &\in \Big\langle\: v^\lambda_\Stab\:\:\Big|\:
\Stab \in {\textstyle\bigcup_{\mu \leq \nu \supset\lambda}} \Tab(\nu), \bi^\Stab = s_r\cdot\bi \Big\rangle.\label{yrprop2}
\end{align}
Moreover 
we will have that $y_r^2 = 0$,
hence it makes sense
to view the power series $p_r(\bi)$ and $q_r(\bi)$ from (\ref{pdef})--(\ref{qdef}) as endomorphisms of $Y(\lambda)$.
Then the generators 
of $H_d^{p,q}$ act by the formulae
\begin{align}\label{f1}
T_r v^\lambda_\T &:=  (\psi_r q_r(\bi) - p_r(\bi)) v^\lambda_\T,\\
L_r v^\lambda_\T &:= 
\left\{
\begin{array}{ll}
(y_r+i_r) v_\T^\lambda&\text{if $\xi = 1$,}\\
\xi^{i_r}(1-y_r) v_\T^\lambda\qquad\qquad&\text{if $\xi \neq 1$,}
\end{array}\right.\label{f2}
\end{align}
which should be compared with (\ref{ysfcomp}) and (\ref{wb}) in the level one case.
Amongst other things, the following theorem justifies the 
terminology ``Young module.''

\begin{Theorem}\label{qh}
The endomorphisms (\ref{f1})--(\ref{f2}) satisfy the defining relations
of $H_d^{p,q}$, so make
$Y := \bigoplus_{\lambda \Vdash d} Y(\lambda)$ into an $H_d^{p,q}$-module.
Let $K_d^{p,q} := \End_{H_d^{p,q}}\left(Y \right)^{\operatorname{op}}$
and $e_\lambda \in K_d^{p,q}$ be the projection of $Y$ onto the summand $Y(\lambda)$.
Then $K_d^{p,q}$
is a basic quasi-hereditary algebra with weight poset
$\{\lambda \Vdash d\}$ partially ordered by $\geq$, and
projective indecomposable modules $P(\lambda) := K_d^{p,q} e_\lambda$,
standard modules $V(\lambda)$ and irreducible modules $L(\lambda)$
for $\lambda \Vdash d$.
Moreover:
\begin{itemize}
\item[(1)]
The left 
$K^{p,q}_d$-module 
$T := \hom_{K^{p,q}_d}(Y, K^{p,q}_d)$
is a projective-injective generator for 
the category
$K^{p,q}_d\operatorname{-mod}$ of finite dimensional left $K^{p,q}_d$-modules,
i.e. it is both projective and injective and every finite dimensional
projective-injective $K^{p,q}_d$-module is isomorphic to a summand of a direct sum of copies of $T$.
\item[(2)]
The following double centralizer property holds:
the natural right action of
$H^{p,q}_d$ on $T$ induces an algebra isomorphism
$H^{p,q}_d \stackrel{\sim}{\rightarrow}
\End_{K^{p,q}_d}(T)^{\op}.$
\item[(3)]
The exact Schur functor $\pi := \hom_{K^{p,q}_d}(T, ?):K^{p,q}_d\operatorname{-mod}
\rightarrow H^{p,q}_d\operatorname{-mod}$
is fully faithful on projective objects.
Hence $K^{p,q}_d$ is a quasi-hereditary cover of $H^{p,q}_d$.
\item[(4)]
For each $\lambda \Vdash d$, 
we have that $\pi P(\lambda) \cong Y(\lambda)$ and 
$\pi V(\lambda)\cong S(\lambda)$.
Moreover if $\lambda$ is restricted then
$\pi L(\lambda) \cong D(\lambda)$,
hence $Y(\lambda)$ is the projective cover of $D(\lambda)$;
if $\lambda$ is not restricted 
then $\pi L(\lambda) = \bz$.
\end{itemize}
\end{Theorem}

By 
the general theory of quasi-hereditary algebras,
the projective indecomposable module $P(\lambda)$ in
Theorem~\ref{qh} has a filtration whose sections are standard modules
with $V(\lambda)$ appearing at the top.
Applying the Schur functor $\pi$ from Theorem~\ref{qh}(3), we get a filtration
of the Young module $Y(\lambda)$ whose sections are Specht modules with $S(\lambda)$ at the top.
The next theorem explains how to see this filtration explicitly in terms of the
orthogonal basis; cf. (\ref{yrprop})--(\ref{yrprop2}).

\begin{Theorem}\label{sb}
For any $\lambda \Vdash d$, the Specht module $S(\lambda)$ is isomorphic
to the quotient of $Y(\lambda)$ by the submodule
$\Big\langle \:v_\T^\lambda\:\:\Big|\:\T \in \bigcup_{\lambda \neq \mu \supset \lambda} \Tab(\mu)\Big\rangle.$
Hence $S(\lambda)$ has a distinguished basis 
$\{v_\T \:|\:\T \in \Tab(\lambda)\}$ arising from the images of
the elements $\{v_\T^\lambda\:|\:\Tab(\lambda)\}$,
on which the actions of the generators of $H_d^{p,q}$ can be computed explicitly
via (\ref{f1})--(\ref{f2}).
Moreover if we let $\mu_1,\dots,\mu_n$ be all the 
$\mu \supset \lambda$ ordered so that $\mu_i \geq \mu_j \Rightarrow 
i \leq j$ and set
$M_j := \Big\langle\:v_\T^\lambda\:\:\Big|\:\T\in
\bigcup_{i \leq j}\Tab(\mu_i)\Big\rangle$,
we get a filtration
$$
\bz = M_0 \subset M_1 \subset \cdots \subset M_n = Y(\lambda)
$$
such that the map $M_j / M_{j-1} \rightarrow S(\mu_j)$
sending $v^\lambda_\T + M_{j-1} \mapsto v_\T$ for $\T \in \Tab(\mu_j)$
is an $H_d^{p,q}$-module isomorphism.
\end{Theorem}

The basis $\{v_\T\:|\:\T \in \Tab(\lambda)\}$ for $S(\lambda)$
arising from Theorem~\ref{sb} is very special.
For example if $\lambda$ is restricted, it contains a basis for the
kernel of the homomorphism $S(\lambda) \twoheadrightarrow
D(\lambda)$,
so that we also get induced an equally explicit basis for the irreducible module $D(\lambda)$. In order to explain this precisely, and also for use when we
define the endomorphisms
$y_r,\psi_r \in \End_{F}(Y(\lambda))$ at the end of the section, we need one more combinatorial excursion.

Suppose we are given a standard $\lambda$-tableau
$\T$ for some $\lambda \Vdash d$.
We are going to represent $\T$ by a new sort of diagram which we call
a {\em stretched $\lambda$-cup diagram}.
To make the translation, let $\bi := \bi^\T$ and $\varnothing = \lambda_0,\lambda_1,\dots,\lambda_{d-1},\lambda_d = \lambda$
be the sequence of bipartitions such that $\lambda_i$ is the shape of the
standard tableau obtained from $\T$ by removing all the boxes containing the entries $\geq (i+1)$.
Stack the weight diagrams of the bipartitions 
$\lambda_0,\lambda_1,\dots,\lambda_d$
in order from bottom to top, and observe
that the weight diagrams 
$\lambda_{r-1}$ and $\lambda_r$ only differ at vertices $i_r$ and $i_{r}+1$.
For each $r=1,\dots,d$, insert vertical line segments connecting
all vertices $< i_r$ or $> (i_r+1)$ that are labelled $\down$ or $\up$ in
$\lambda_{r-1}$ and $\lambda_r$. Then
connect the remaining vertices $i_r$ and $i_r+1$
of $\lambda_{r-1}$ and $\lambda_r$ as in the appropriate one of the following pictures:
$$
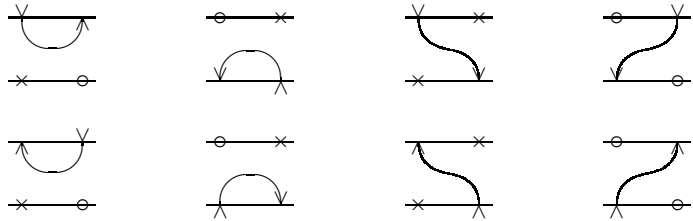
\begin{picture}(75,78)
\put(6,72){\line(1,0){33}}
\put(6,48){\line(1,0){33}}
\put(22.5,72){\oval(23,23)[b]}
\put(7.7,46.1){$\cross$}
\put(31.3,45.4){$\circ$}
\put(8.2,72.3){{$\down$}}
\put(31.2,67.3){{$\up$}}

\put(6,25){\line(1,0){33}}
\put(6,1){\line(1,0){33}}
\put(22.5,25){\oval(23,23)[b]}
\put(7.7,-.9){$\cross$}
\put(31.3,-1.6){$\circ$}
\put(31.2,25.3){{$\down$}}
\put(8.2,20.3){{$\up$}}
\end{picture}
\begin{picture}(75,78)
\put(6,72){\line(1,0){33}}
\put(6,48){\line(1,0){33}}
\put(22.5,48){\oval(23,23)[t]}
\put(30.7,70.1){$\cross$}
\put(8.3,69.4){$\circ$}
\put(8.2,48.3){{$\down$}}
\put(31.2,43.3){{$\up$}}

\put(6,25){\line(1,0){33}}
\put(6,1){\line(1,0){33}}
\put(22.5,1){\oval(23,23)[t]}
\put(30.7,23.1){$\cross$}
\put(8.3,22.4){$\circ$}
\put(31.2,1.3){{$\down$}}
\put(8.2,-3.7){{$\up$}}
\end{picture}
\begin{picture}(75,78)
\put(6,72){\line(1,0){33}}
\put(6,48){\line(1,0){33}}
\qbezier(22.5,60)(11,62)(11,72)
\qbezier(34,48)(34,58)(22.5,60)
\put(30.7,70.1){$\cross$}
\put(7.7,46.1){$\cross$}
\put(8.2,72.3){{$\down$}}
\put(31.2,48.3){{$\down$}}

\put(6,25){\line(1,0){33}}
\put(6,1){\line(1,0){33}}
\put(30.7,23.1){$\cross$}
\put(7.7,-.9){$\cross$}
\qbezier(22.5,13)(11,15)(11,25)
\qbezier(34,1)(34,11)(22.5,13)
\put(31.2,-3.7){{$\up$}}
\put(8.2,20.3){{$\up$}}
\end{picture}
\begin{picture}(45,78)
\put(6,72){\line(1,0){33}}
\put(6,48){\line(1,0){33}}
\qbezier(22.5,60)(11,58)(11,48)
\qbezier(34,72)(34,62)(22.5,60)
\put(8.3,69.4){$\circ$}
\put(31.3,45.4){$\circ$}
\put(8.2,48.3){{$\down$}}
\put(31.2,72.3){{$\down$}}

\put(6,25){\line(1,0){33}}
\put(6,1){\line(1,0){33}}
\put(8.3,22.4){$\circ$}
\put(31.3,-1.6){$\circ$}
\qbezier(22.5,13)(11,11)(11,1)
\qbezier(34,25)(34,15)(22.5,13)
\put(31.2,20.3){{$\up$}}
\put(8.2,-3.7){{$\up$}}
\end{picture}
$$
See Table~\ref{mytable} for some examples.
\begin{table}
$$
\begin{array}{|c|c|c|c|}
\hline
\T_1&\T_2&\T_3&\T_4\\&&&\\
\left(\begin{array}{l}\diagram{1\cr}\\\vspace{0.6mm}\\\end{array},
\begin{array}{l}
\diagram{2&4\cr3\cr}\\
\end{array}
\right)
&
\left(\begin{array}{l}\diagram{2\cr}\\\vspace{0.6mm}\\\end{array},
\begin{array}{l}
\diagram{1&4\cr3\cr}\\
\end{array}
\right)
&
\left(\begin{array}{l}\diagram{3\cr}\\\vspace{0.6mm}\\\end{array},
\begin{array}{l}
\diagram{1&4\cr2\cr}\\
\end{array}
\right)
&
\left(\begin{array}{l}\diagram{4\cr}\\\vspace{0.6mm}\\\end{array},
\begin{array}{l}
\diagram{1&3\cr2\cr}\\
\end{array}
\right)\\&&&\\
\begin{picture}(75,90)
\put(.3,140.3){$_{0}$}\put(4.6,136){$_{\smallsetminus}$}
\put(28,140.3){$_{0}$}\put(32.3,136){$_{\smallsetminus}$}
\put(0,-7){$_{-1}$}
\put(25.4,-7){$_{0}$}
\put(45.6,-7){$_{1}$}
\put(65.8,-7){$_{2}$}
\put(3,82){\line(1,0){69}}
\put(3,62){\line(1,0){69}}
\put(3,42){\line(1,0){69}}
\put(3,22){\line(1,0){69}}
\put(3,2){\line(1,0){69}}

\put(5,82.1){$\down$}
\put(45,82.1){$\down$}
\put(25,77.3){$\up$}
\put(65,77.3){$\up$}

\put(5,62.1){$\down$}
\put(25,57.3){$\up$}
\put(65.1,59.3){$\circ$}
\put(44.5,60.1){$\cross$}

\put(25.1,39.3){$\circ$}
\put(65.1,39.3){$\circ$}
\put(4.5,40.1){$\cross$}
\put(44.5,40.1){$\cross$}

\put(45,22.1){$\down$}
\put(25,17.7){$\up$}
\put(65.1,19.3){$\circ$}
\put(4.5,20.1){$\cross$}

\put(45.1,-.7){$\circ$}
\put(65.1,-.7){$\circ$}
\put(4.5,.1){$\cross$}
\put(24.5,.1){$\cross$}

\put(7.8,82){\line(0,-1){20}}
\put(27.8,82){\line(0,-1){20}}
\put(57.9,82){\oval(20,20)[b]}
\put(17.9,62){\oval(20,20)[b]}
\put(37.9,22){\oval(20,20)[t]}
\put(37.9,22){\oval(20,20)[b]}
\end{picture}
&
\begin{picture}(75,90)
\put(.3,140.3){$_{0}$}\put(4.6,136){$_{\smallsetminus}$}
\put(28,140.3){$_{0}$}\put(32.3,136){$_{\smallsetminus}$}
\put(0,-7){$_{-1}$}
\put(25.4,-7){$_{0}$}
\put(45.6,-7){$_{1}$}
\put(65.8,-7){$_{2}$}
\put(3,82){\line(1,0){69}}
\put(3,62){\line(1,0){69}}
\put(3,42){\line(1,0){69}}
\put(3,22){\line(1,0){69}}
\put(3,2){\line(1,0){69}}

\put(5,82.1){$\down$}
\put(45,82.1){$\down$}
\put(25,77.3){$\up$}
\put(65,77.3){$\up$}

\put(5,62.1){$\down$}
\put(25,57.3){$\up$}
\put(65.1,59.3){$\circ$}
\put(44.5,60.1){$\cross$}

\put(25.1,39.3){$\circ$}
\put(65.1,39.3){$\circ$}
\put(4.5,40.1){$\cross$}
\put(44.5,40.1){$\cross$}

\put(25,22.1){$\down$}
\put(45,17.3){$\up$}
\put(65.1,19.3){$\circ$}
\put(4.5,19.1){$\cross$}

\put(45.1,-0.7){$\circ$}
\put(65.1,-0.7){$\circ$}
\put(4.5,.1){$\cross$}
\put(24.5,.1){$\cross$}

\put(7.8,82){\line(0,-1){20}}
\put(27.8,82){\line(0,-1){20}}
\put(57.9,82){\oval(20,20)[b]}
\put(17.9,62){\oval(20,20)[b]}
\put(37.9,22){\oval(20,20)[t]}
\put(37.9,22){\oval(20,20)[b]}
\end{picture}
&
\begin{picture}(75,90)
\put(.3,140.3){$_{0}$}\put(4.6,136){$_{\smallsetminus}$}
\put(28,140.3){$_{0}$}\put(32.3,136){$_{\smallsetminus}$}
\put(0,-7){$_{-1}$}
\put(25.4,-7){$_{0}$}
\put(45.6,-7){$_{1}$}
\put(65.8,-7){$_{2}$}
\put(3,82){\line(1,0){69}}
\put(3,62){\line(1,0){69}}
\put(3,42){\line(1,0){69}}
\put(3,22){\line(1,0){69}}
\put(3,2){\line(1,0){69}}

\put(5,82.1){$\down$}
\put(45,82.1){$\down$}
\put(25,77.3){$\up$}
\put(65,77.3){$\up$}

\put(5,62.1){$\down$}
\put(25,57.3){$\up$}
\put(65.1,59.3){$\circ$}
\put(44.5,60.1){$\cross$}

\put(5,42.1){$\down$}
\put(45,37.3){$\up$}
\put(65.1,39.3){$\circ$}
\put(24.5,40.1){$\cross$}

\put(25,22.1){$\down$}
\put(45,17.3){$\up$}
\put(65.1,19.3){$\circ$}
\put(4.5,20.1){$\cross$}

\put(45.1,-0.7){$\circ$}
\put(65.1,-0.7){$\circ$}
\put(4.5,.1){$\cross$}
\put(24.5,.1){$\cross$}

\put(7.8,82){\line(0,-1){40}}
\put(27.8,82){\line(0,-1){20}}
\put(47.8,42){\line(0,-1){20}}
\put(57.9,82){\oval(20,20)[b]}
\put(37.9,22){\oval(20,20)[b]}
\qbezier(7.8,42)(7.8,34)(17.8,32)
\qbezier(17.8,32)(27.8,30)(27.8,22)
\qbezier(27.8,62)(27.8,54)(37.8,52)
\qbezier(37.8,52)(47.8,50)(47.8,42)
\end{picture}
&
\begin{picture}(75,90)
\put(.3,140.3){$_{0}$}\put(4.6,136){$_{\smallsetminus}$}
\put(28,140.3){$_{0}$}\put(32.3,136){$_{\smallsetminus}$}
\put(0,-7){$_{-1}$}
\put(25.4,-7){$_{0}$}
\put(45.6,-7){$_{1}$}
\put(65.8,-7){$_{2}$}
\put(3,82){\line(1,0){69}}
\put(3,62){\line(1,0){69}}
\put(3,42){\line(1,0){69}}
\put(3,22){\line(1,0){69}}
\put(3,2){\line(1,0){69}}

\put(5,82.1){$\down$}
\put(45,82.1){$\down$}
\put(25,77.3){$\up$}
\put(65,77.3){$\up$}

\put(5,62.1){$\down$}
\put(65,57.3){$\up$}
\put(45.1,59.3){$\circ$}
\put(24.5,60.1){$\cross$}

\put(5,42.1){$\down$}
\put(45,37.3){$\up$}
\put(65.1,39.3){$\circ$}
\put(24.5,40.1){$\cross$}

\put(25,22.1){$\down$}
\put(45,17.7){$\up$}
\put(65.1,19.3){$\circ$}
\put(4.5,20.1){$\cross$}

\put(45.1,-0.7){$\circ$}
\put(65.1,-0.7){$\circ$}
\put(4.5,.1){$\cross$}
\put(24.5,.1){$\cross$}

\put(7.8,82){\line(0,-1){40}}
\put(67.8,82){\line(0,-1){20}}
\put(47.8,42){\line(0,-1){20}}
\put(37.9,82){\oval(20,20)[b]}
\put(37.9,22){\oval(20,20)[b]}
\qbezier(7.8,42)(7.8,34)(17.8,32)
\qbezier(17.8,32)(27.8,30)(27.8,22)
\qbezier(67.8,62)(67.8,54)(57.8,52)
\qbezier(57.8,52)(47.8,50)(47.8,42)
\end{picture}\\&&&\\
\hline\hline
\T_5&\T_6&\T_7&\T_8\\&&&\\
\left(\begin{array}{l}\diagram{1\cr}\\\vspace{0.6mm}\\\end{array},
\begin{array}{l}
\diagram{2&3\cr4\cr}\\
\end{array}
\right)
&
\left(\begin{array}{l}\diagram{2\cr}\\\vspace{0.6mm}\\\end{array},
\begin{array}{l}
\diagram{1&3\cr4\cr}\\
\end{array}
\right)
&
\left(\begin{array}{l}\diagram{3\cr}\\\vspace{0.6mm}\\\end{array},
\begin{array}{l}
\diagram{1&2\cr4\cr}\\
\end{array}
\right)
&
\left(\begin{array}{l}\diagram{4\cr}\\\vspace{0.6mm}\\\end{array},
\begin{array}{l}
\diagram{1&2\cr3\cr}\\
\end{array}
\right)\\&&&\\
\begin{picture}(75,90)
\put(0.3,140.3){$_{0}$}\put(4.6,136){$_{\smallsetminus}$}
\put(28,140.3){$_{0}$}\put(32.3,136){$_{\smallsetminus}$}
\put(0,-7){$_{-1}$}
\put(25.4,-7){$_{0}$}
\put(45.6,-7){$_{1}$}
\put(65.8,-7){$_{2}$}
\put(3,82){\line(1,0){69}}
\put(3,62){\line(1,0){69}}
\put(3,42){\line(1,0){69}}
\put(3,22){\line(1,0){69}}
\put(3,2){\line(1,0){69}}

\put(5,82.1){$\down$}
\put(45,82.1){$\down$}
\put(25,77.3){$\up$}
\put(65,77.3){$\up$}

\put(45,62.1){$\down$}
\put(65,57.3){$\up$}
\put(25.1,59.3){$\circ$}
\put(4.5,60.1){$\cross$}

\put(25.1,39.3){$\circ$}
\put(65.1,39.3){$\circ$}
\put(4.5,40.1){$\cross$}
\put(44.5,40.1){$\cross$}

\put(45,22.1){$\down$}
\put(25,17.7){$\up$}
\put(65.1,19.3){$\circ$}
\put(4.5,20.1){$\cross$}

\put(45.1,-.7){$\circ$}
\put(65.1,-.7){$\circ$}
\put(4.5,.1){$\cross$}
\put(24.5,.1){$\cross$}

\put(47.8,82){\line(0,-1){20}}
\put(67.8,82){\line(0,-1){20}}
\put(57.9,62){\oval(20,20)[b]}
\put(17.9,82){\oval(20,20)[b]}
\put(37.9,22){\oval(20,20)[t]}
\put(37.9,22){\oval(20,20)[b]}
\end{picture}
&
\begin{picture}(75,90)
\put(.3,140.3){$_{0}$}\put(4.6,136){$_{\smallsetminus}$}
\put(28,140.3){$_{0}$}\put(32.3,136){$_{\smallsetminus}$}
\put(0,-7){$_{-1}$}
\put(25.4,-7){$_{0}$}
\put(45.6,-7){$_{1}$}
\put(65.8,-7){$_{2}$}
\put(3,82){\line(1,0){69}}
\put(3,62){\line(1,0){69}}
\put(3,42){\line(1,0){69}}
\put(3,22){\line(1,0){69}}
\put(3,2){\line(1,0){69}}

\put(5,82.1){$\down$}
\put(45,82.1){$\down$}
\put(25,77.3){$\up$}
\put(65,77.3){$\up$}

\put(45,62.1){$\down$}
\put(65,57.3){$\up$}
\put(25.1,59.3){$\circ$}
\put(4.5,60.1){$\cross$}

\put(25.1,39.3){$\circ$}
\put(65.1,39.3){$\circ$}
\put(4.5,40.1){$\cross$}
\put(44.5,40.1){$\cross$}

\put(25,22.1){$\down$}
\put(45,17.3){$\up$}
\put(65.1,19.3){$\circ$}
\put(4.5,19.1){$\cross$}

\put(45.1,-0.7){$\circ$}
\put(65.1,-0.7){$\circ$}
\put(4.5,.1){$\cross$}
\put(24.5,.1){$\cross$}

\put(47.8,82){\line(0,-1){20}}
\put(67.8,82){\line(0,-1){20}}
\put(57.9,62){\oval(20,20)[b]}
\put(17.9,82){\oval(20,20)[b]}
\put(37.9,22){\oval(20,20)[t]}
\put(37.9,22){\oval(20,20)[b]}
\end{picture}
&
\begin{picture}(75,90)
\put(.3,140.3){$_{0}$}\put(4.6,136){$_{\smallsetminus}$}
\put(28,140.3){$_{0}$}\put(32.3,136){$_{\smallsetminus}$}
\put(0,-7){$_{-1}$}
\put(25.4,-7){$_{0}$}
\put(45.6,-7){$_{1}$}
\put(65.8,-7){$_{2}$}
\put(3,82){\line(1,0){69}}
\put(3,62){\line(1,0){69}}
\put(3,42){\line(1,0){69}}
\put(3,22){\line(1,0){69}}
\put(3,2){\line(1,0){69}}

\put(5,82.1){$\down$}
\put(45,82.1){$\down$}
\put(25,77.3){$\up$}
\put(65,77.3){$\up$}

\put(45,62.1){$\down$}
\put(65,57.3){$\up$}
\put(25.1,59.3){$\circ$}
\put(4.5,60.1){$\cross$}

\put(25,42.1){$\down$}
\put(65,37.3){$\up$}
\put(45.1,39.3){$\circ$}
\put(4.5,40.1){$\cross$}

\put(25,22.1){$\down$}
\put(45,17.3){$\up$}
\put(65.1,19.3){$\circ$}
\put(4.5,20.1){$\cross$}

\put(45.1,-0.7){$\circ$}
\put(65.1,-0.7){$\circ$}
\put(4.5,.1){$\cross$}
\put(24.5,.1){$\cross$}

\put(67.8,82){\line(0,-1){40}}
\put(47.8,82){\line(0,-1){20}}
\put(27.8,42){\line(0,-1){20}}
\put(17.9,82){\oval(20,20)[b]}
\put(37.9,22){\oval(20,20)[b]}

\qbezier(47.8,62)(47.8,54)(37.8,52)
\qbezier(37.8,52)(27.8,50)(27.8,42)

\qbezier(67.8,42)(67.8,34)(57.8,32)
\qbezier(57.8,32)(47.8,30)(47.8,22)
\end{picture}
&
\begin{picture}(75,90)
\put(.3,140.3){$_{0}$}\put(4.6,136){$_{\smallsetminus}$}
\put(28,140.3){$_{0}$}\put(32.3,136){$_{\smallsetminus}$}
\put(0,-7){$_{-1}$}
\put(25.4,-7){$_{0}$}
\put(45.6,-7){$_{1}$}
\put(65.8,-7){$_{2}$}
\put(3,82){\line(1,0){69}}
\put(3,62){\line(1,0){69}}
\put(3,42){\line(1,0){69}}
\put(3,22){\line(1,0){69}}
\put(3,2){\line(1,0){69}}

\put(5,82.1){$\down$}
\put(45,82.1){$\down$}
\put(25,77.3){$\up$}
\put(65,77.3){$\up$}

\put(5,62.1){$\down$}
\put(65,57.3){$\up$}
\put(45.1,59.3){$\circ$}
\put(24.5,60.1){$\cross$}

\put(25,42.1){$\down$}
\put(65,37.3){$\up$}
\put(45.1,39.3){$\circ$}
\put(4.5,40.1){$\cross$}

\put(25,22.1){$\down$}
\put(45,17.7){$\up$}
\put(65.1,19.3){$\circ$}
\put(4.5,20.1){$\cross$}

\put(45.1,-0.7){$\circ$}
\put(65.1,-0.7){$\circ$}
\put(4.5,.1){$\cross$}
\put(24.5,.1){$\cross$}

\put(7.8,82){\line(0,-1){20}}
\put(27.8,42){\line(0,-1){20}}
\put(67.8,82){\line(0,-1){40}}
\put(37.9,82){\oval(20,20)[b]}
\put(37.9,22){\oval(20,20)[b]}
\qbezier(7.8,62)(7.8,54)(17.8,52)
\qbezier(17.8,52)(27.8,50)(27.8,42)
\qbezier(67.8,42)(67.8,34)(57.8,32)
\qbezier(57.8,32)(47.8,30)(47.8,22)
\end{picture}\\&&&\\
\hline
\end{array}
$$
\caption{
For
$\la = ((1),(21))$ and $p=q=0$,
this table displays the stretched $\lambda$-cup diagrams
corresponding to the eight standard $\la$-tableaux, which are denoted
$\T_1,\dots,\T_8$.}
\label{mytable}
\end{table}

Ignoring the weight diagrams themselves, the stretched cup diagram 
of any $\T \in \Tab(\lambda)$
decomposes into various connected components: {\em circles} 
in the interior of the diagram, {\em boundary cups} 
whose endpoints are vertices on the top number line, and 
{\em line segments} which stretch between the bottom and top number lines.
The top weight diagram $\lambda$ gives an orientation to each 
of the boundary cups, either counter-clockwise or clockwise.
We define the {\em weight} $\wt(\T)$ 
to be the bipartition whose weight diagram is obtained from 
$\lambda$ by switching 
the labels at the ends of each of the clockwise boundary cups.
Also for $i \geq 0$
let $\Tab_i(\lambda)$ denote the set of all $\T \in \Tab(\lambda)$
such that the corresponding stretched cup diagram has exactly $i$
clockwise boundary cups.
In particular,
\begin{equation}
\Tab_0(\lambda) = \{\T \in \Tab(\lambda)\:|\:\wt(\T) = \lambda\},
\end{equation}
which is non-empty if and only if $\lambda$ is restricted.

\begin{Theorem}\label{sf}
Given a restricted $\lambda \Vdash d$, 
the irreducible module $D(\lambda)$ is isomorphic to the quotient
of $S(\lambda)$ by the submodule
$\big\langle \,v_\T\:\big|\:\T \in \bigcup_{i \geq 1} \Tab_i(\lambda)\,\big\rangle$.
Hence $D(\lambda)$ has a distinguished basis $\{\bar{v}_\T\:|\:\T \in \Tab_0(\lambda)\}$ arising from the images of the elements $\{v_\T\:|\:\T \in \Tab_0(\lambda)\}$, on which the actions of the generators of $H_d^{p,q}$
can be computed explicitly via (\ref{f1})--(\ref{f2}).
Moreover given an arbitrary $\lambda \Vdash d$,
let $N_j := \langle \, v_\T\:\big|\:\T \in \bigcup_{i \geq j} \Tab_i(\lambda)\,\big\rangle$.
Then
$$
S(\lambda)  = N_0 \supseteq N_1 \supseteq \cdots
$$
is a filtration of $S(\lambda)$ 
such that 
$N_j / N_{j+1} \cong \bigoplus_\mu D(\mu)$,
direct sum over all $\mu \Vdash d$ such that
there is a $\lambda$-cap 
diagram of weight $\mu$ with exactly $j$ clockwise caps;
the explicit 
isomorphism here sends $v_\T + N_{j+1} \in N_j / N_{j+1}$ for $\T \in \Tab_j(\lambda)$
to $\bar v_{\bar \T} \in D(\wt(\T))$ where 
$\bar\T$ is the standard tableau whose stretched cup diagram
is obtained from that of $\T$ by reversing the labels on all clockwise 
boundary cups.
\end{Theorem}

It just remains to explain the definitions of 
$y_r, \psi_r \in \End_F(Y(\lambda))$.
Continue with $\lambda \Vdash d$.
Fix $\T \in \Tab(\mu)$ for some $\mu \supset \lambda$,
hence a basis vector $v_\T^\lambda \in Y(\lambda)$.
Let $\bi := \bi^\T$. 
To start with we take care of some awkward signs:
we will actually define 
$\bar y_r, \bar \psi_r \in \End_F(Y(\lambda))$,
and then $y_r$ and $\psi_r$ are related to these by the formulae
\begin{equation}
y_r v_\T^\lambda = \sigma_r(\bi) \bar y_r v^\lambda_\T,
\qquad
\psi_r v^\lambda_\T = \left\{
\begin{array}{ll}
- \sigma_r(\bi) \bar\psi_r v^\lambda_\T&\text{if }i_{r+1} \in \{i_r,i_r+1\},\\
\bar\psi_r v_\T&\text{otherwise,}
\end{array}
\right.
\end{equation}
where
$\sigma_r(\bi) := (-1)^{\min(p,i_r)+\min(q,i_r)+\delta_{i_1,i_r}+\cdots+\delta_{i_{r-1},i_r}}$.

To calculate $\bar y_r v_\T^\lambda$, 
let 
$\Tstacked{\lambda}$ denote the composite diagram obtained
by gluing the stretched $\mu$-cup diagram corresponding to $\T$
under the unique $\mu$-cap diagram of weight $\lambda$.
We refer to the horizontal strips between the number lines 
in this diagram as its {\em layers}, and index them
by $1,\dots,d$ in order from bottom to top.
There is a unique connected component in the diagram
$\Tstacked{\lambda}$
which is non-trivial in the $r$th layer, i.e. its intersection with
the $r$th layer involves something other than
vertical line segments.
If this connected component is a counter-clockwise circle, we
reverse all the labels
$\down$ or $\up$
on the component to get
a new diagram of the form $\Sstacked{\lambda}$ for a unique
standard tableau $\Stab$, then set $\bar y_r v_\T^\lambda := v_\Stab^\lambda$; otherwise we simply set $\bar y_r v_\T^\lambda := 0$.
For example, in the notation of Table~\ref{mytable}
taking $\lambda = ((1),(21))$, we have
$y_1 v_{\T_2}^\lambda = v_{\T_1}^\lambda$,
and $y_4 v_{\T_4}^\lambda = -v_{\Stab}^\lambda$
where 
$\Stab = (\T_4^R,\T_4^L)$.

To calculate $\bar\psi_r v_\T^\lambda$, there are three cases. The easiest is when $|i_r-i_{r+1}| > 1$,
when $s_r \cdot \T$ is again a standard tableau as in the level one
case and we set
\begin{equation}
\psi_r v_\T^\lambda = \bar \psi_r v_\T^\lambda :=
v_{s_r \cdot \T}^\lambda. 
\end{equation}
In terms of $\Tstacked{\lambda}$,
this corresponds to sliding the parts of the diagram that are non-trivial in
layers $r$ and $(r+1)$ past each other. For example in 
the notation from 
Table~\ref{mytable} again
we have that
$\psi_2 v_{\T_4}^\lambda = v_{\T_8}^\lambda$ 
and
$\psi_3 v_{\T_2}^\lambda = v_{\T_6}^\lambda$.

Next suppose that $i_r = i_{r+1}$. Then the diagram
$\Tstacked{\lambda}$ has a small circle in layers $r$ and $r+1$. If
this circle is counter-clockwise we set $\bar\psi_r v^\lambda_\T :=
0$;
otherwise the circle is clockwise and we let $\bar\psi_r v^\lambda_\T := v^\lambda_\Stab$ where
$\Sstacked{\lambda}$ is obtained from $\Tstacked{\lambda}$
by reversing the labels on this circle.
For example $\psi_1
v_{\T_1}^\lambda = -v_{\T_2}^\lambda$ and $\psi_1 v_{\T_2}^\lambda
= 0$.

Finally suppose that $|i_r-i_{r+1}| = 1$.
If there are no standard tableaux with
residue sequence $s_r \cdot \bi$, we simply set $\bar\psi_r
v_\T^\lambda := 0$. If there is at least one such standard tableau,
the part of the diagram
$\Tstacked{\lambda}$
that is non-trivial in layers $r$ and $(r+1)$ matches one of the
following eight configurations:
$$
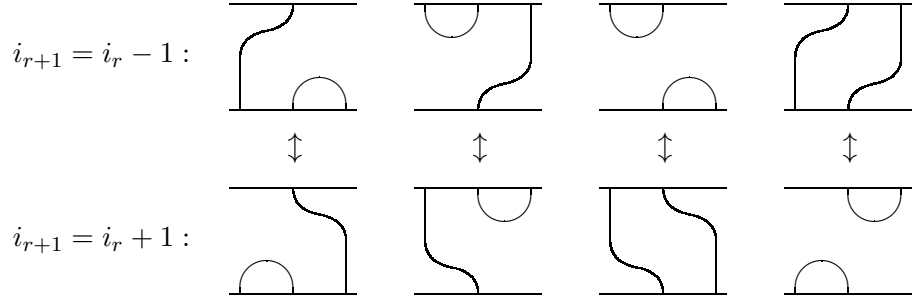
\begin{picture}(10,47)
\put(-164,20){$i_{r+1}=i_r-1:$}

\put(-82,42){\line(1,0){48}}
\put(-82,2){\line(1,0){48}}
\put(-78,22){\line(0,-1){20}}
\put(-48,2){\oval(20,25)[t]}
\qbezier(-58,42)(-58,34)(-68,32)
\qbezier(-68,32)(-78,30)(-78,22)
\put(-60,-15){$\updownarrow$}

\put(-12,42){\line(1,0){48}}
\put(-12,2){\line(1,0){48}}
\put(2,42){\oval(20,25)[b]}
\put(32,42){\line(0,-1){20}}
\qbezier(32,22)(32,14)(22,12)
\qbezier(22,12)(12,10)(12,2)
\put(10,-15){$\updownarrow$}

\put(58,42){\line(1,0){48}}
\put(58,2){\line(1,0){48}}
\put(72,42){\oval(20,25)[b]}
\put(92,2){\oval(20,25)[t]}
\put(80,-15){$\updownarrow$}

\put(128,42){\line(1,0){48}}
\put(128,2){\line(1,0){48}}
\qbezier(152,42)(152,34)(142,32)
\qbezier(142,32)(132,30)(132,22)
\put(172,42){\line(0,-1){20}}
\put(132,22){\line(0,-1){20}}
\qbezier(172,22)(172,14)(162,12)
\qbezier(162,12)(152,10)(152,2)
\put(150,-15){$\updownarrow$}
\end{picture}
$$
$$
\begin{picture}(10,64)
\put(-164,20){$i_{r+1}=i_r+1:$}

\put(-82,42){\line(1,0){48}}
\put(-82,2){\line(1,0){48}}
\qbezier(-58,42)(-58,34)(-48,32)
\qbezier(-48,32)(-38,30)(-38,22)
\put(-38,22){\line(0,-1){20}}
\put(-68,2){\oval(20,25)[t]}

\put(-12,42){\line(1,0){48}}
\put(-12,2){\line(1,0){48}}
\put(22,42){\oval(20,25)[b]}
\put(-8,42){\line(0,-1){20}}
\qbezier(-8,22)(-8,14)(2,12)
\qbezier(2,12)(12,10)(12,2)

\put(58,42){\line(1,0){48}}
\put(58,2){\line(1,0){48}}
\qbezier(82,42)(82,34)(92,32)
\qbezier(92,32)(102,30)(102,22)
\put(62,42){\line(0,-1){20}}
\put(102,22){\line(0,-1){20}}
\qbezier(62,22)(62,14)(72,12)
\qbezier(72,12)(82,10)(82,2)

\put(128,42){\line(1,0){48}}
\put(128,2){\line(1,0){48}}
\put(162,42){\oval(20,25)[b]}
\put(142,2){\oval(20,25)[t]}
\end{picture}
$$
The part of the diagram just displayed can belong to either one or two connected components
in the larger diagram
$\Tstacked{\lambda}$. In the former case we define the type
to be $1, x$ or $y$ according to whether the connected component is a
counter-clockwise circle, a clockwise circle or a line
segment;
in the latter case we define the type to be
$1 \otimes 1, 1 \otimes x, 1 \otimes y, x \otimes x$, $x \otimes y$
or $y \otimes y$
according to whether there are two counter-clockwise circles, one
counter-clockwise and one clockwise circle, one counter-clockwise
circle and one line segment, two clockwise circles, one
clockwise circle and one line segment, or two line segments.
Erase all the labels from the one or two components, transform the
$r$th and $(r+1)$th layers 
as indicated by the correspondence $\updownarrow$ in the above diagrams, then
finally reintroduce labels into the two or one components created by
this transformation
 to obtain 
some new diagrams of the form $\Sstacked{\lambda}$ for standard tableaux
$\Stab$; then $\bar\psi_r v_\T^\lambda$ is defined to be the sum of the
corresponding
basis vectors $v_\Stab^\lambda$.
The rules to reintroduce labels in the final step here depends on the initial type as follows:
\begin{align}\label{rule1}
1 &\mapsto 1 \otimes x + x \otimes 1,\;\quad\qquad
x \mapsto x \otimes x, \;\quad\qquad y \mapsto x \otimes y,\\
1 \otimes 1 &\mapsto 1,
\:
1 \otimes x \mapsto x,
\:
1 \otimes y \mapsto y,
\:
x \otimes x \mapsto 0,
\:
x \otimes y \mapsto 0,
\:
y \otimes y \mapsto 0,
\label{rule2}
\end{align}
where again $1$ represents a counter-clockwise circle, 
$x$ a clockwise circle and $y$ a line segment.
The first rule in (\ref{rule1}) means that we get two diagrams in which the two
components are
oriented counter-clockwise and clockwise in the first
and vice versa in the second; 
the last three rules in (\ref{rule2}) mean that we get zero; the other
five rules are
interpreted similarly.
For example
$\psi_1 v_{\T_7}^\lambda = 0$,
$\psi_2 v_{\T_2}^\lambda = v_{\T_3}^\lambda$ and
$\psi_3 v_{\T_8}^\lambda = y_3 v_{\T_7}^\lambda - y_4 v_{\T_7}^\lambda$.

\vspace{2mm}
\noindent
{\em Notes.}
In the case that $\xi \neq 1$, the algebra
$H_d^{p,q}$ can obviously be identified with the finite
Iwahori-Hecke algebra 
of type $B_d$ at long root
parameter $\xi$
and short root parameter $-\xi^{q-p}$,
the generator usually denoted $T_0$ in that Iwahori-Hecke algebra 
being $-\xi^{-p} L_1$.
Furthermore by the main result of \cite{BKyoung}
the algebra $H_d^{p,q}$
is isomorphic 
in both the 
non-degenerate and the degenerate cases to the 
cyclotomic {quiver Hecke algebra} of 
\cite{KLa}, \cite{Rou}
for the quiver $A_\infty$ and the level two
weight $\La_p + \La_q$.
The semisimplicity criterion for $H_d^{p,q}$ is due to Dipper and James 
\cite[Theorem 5.5]{DJ}. In all the semisimple cases an analogue of Young's orthogonal form was worked out by Hoefsmit in \cite{H}.
The construction of Specht modules as induced modules originates in work of Vazirani. Theorem~\ref{vazt} is essentially the level two
case of \cite[Theorem 3.4]{Vaz}
if $p \geq q$; it can be proved by similar techniques when $p < q$. 

The cyclotomic quiver Hecke
algebra just mentioned is naturally $\Z$-graded. The
Young module $Y(\lambda)$ can be interpreted as graded module over this
graded algebra, with $\Z$-grading defined so that the basis vector
$v_\T^\lambda$ is of degree equal to the number of $\up$'s or the
number of $\down$'s in the weight diagram of $\lambda$, whichever is
smaller, plus the total number of clockwise circles minus the total
number of counter-clockwise circles in the diagram
$\Tstacked{\lambda}$. 
The grading on $Y(\lambda)$ induces gradings on the quotients
$S(\lambda)$ and (assuming $\lambda$ is restricted) $D(\lambda)$, so
that the basis elements $v_\T$ and $\bar v_\T$ constructed in Theorems~\ref{sb} and \ref{sf} are
homogeneous of degree equal to
\begin{equation}\label{dt}
\deg(\T)
:= 
\#(\text{clockwise cups})
- \#(\text{counter-clockwise caps})
\end{equation}
in the stretched cup diagram associated to $\T$.
In fact with this grading $S(\lambda)$ is isomorphic to the {\em graded Specht module} of
\cite{BKW}, and the definition (\ref{dt}) agrees with \cite[(3.5)]{BKW}.
The filtrations in Theorems~\ref{sb} and \ref{sf} are filtrations of
graded modules; the section $M_j / M_{j-1}$ in Theorem~\ref{sb} is
actually isomorphic as a graded module to $S(\mu_j)\langle d_j
\rangle$ (the graded Specht module $S(\mu_j)$
shifted up in degree by $d_j$) where $d_j$ is the number of clockwise
caps in the unique $\mu_j$-cap diagram of weight $\lambda$; the
section
$N_j / N_{j+1}$ in Theorem~\ref{sf} is isomorphic as a graded module to $\bigoplus_\mu
D(\mu)\langle j \rangle$ summing over $\mu$ as in the statement of the
theorem.

Over fields of characteristic $0$, Theorem~\ref{cm} can be deduced from Ariki's categorification theorem \cite{Ari}, \cite{BKariki}; the necessary combinatorics of canonical bases was worked out by Leclerc and Miyachi \cite{LM} 
in terms of the combinatorics of Lusztig's symbols. 
Our formulation using
cap diagrams is equivalent to this.
The observation that Theorem~\ref{cm} is also valid over fields of positive characteristic was first observed by Ariki and Mathas \cite[Corollary 3.7]{AM}.

In \cite{BS3}, we explained a different approach starting from
the explicit construction of the 
algebra $K_d^{p,q}$ 
(to be explained in section 4 below), and also an alternative construction of the
left $K_d^{p,q}$-module $T$ from Theorem~\ref{qh} in terms of certain
special projective functors $F_i$;  although we worked over the ground field $\mathbb{C}$
the relevant arguments in \cite{BS3} 
can be carried out over arbitrary fields with only minor modifications as
noted in \cite[Remark 8.7]{BS3}.
The precise references needed to extract
Theorem~\ref{qh} from \cite{BS1}--\cite{BS3} are as follows:
(1) follows from the alternative
definition of $T$ from \cite[(6.1)]{BS3} plus \cite[Lemma
6.1]{BS3};
(2) follows from \cite[Corollary 8.6]{BS3} plus the definition
\cite[(6.2)]{BS3};
(3) and (4) apart from the isomorphisms
$\pi P(\lambda) \cong Y(\lambda)$ and $\pi V(\lambda) \cong S(\lambda)$
follow from \cite[Lemma 8.13]{BS3}.
The isomorphism $\pi P(\lambda) \cong Y(\lambda)$
and the explicit construction of $Y(\lambda)$ via the orthogonal form
described above can be deduced from 
\cite[Lemma 6.6]{BS3}; to get the precise formulae 
(\ref{f1})--(\ref{f2}) one needs also to use the isomorphism theorem from
\cite{BKyoung}.
The isomorphism $\pi V(\lambda) \cong S(\lambda)$ 
is established in the degenerate case in \cite[Lemma 9.3 and Corollary 9.6]{BS3}; it can be deduced in the non-degenerate case 
too by a base change argument involving the construction of
\cite{BKW}.
Theorem~\ref{sb} is a consequence 
of \cite[Theorem 5.1]{BS1} on applying the Schur functor. Similarly Theorem~\ref{sf} is a consequence of
\cite[Theorem 5.2]{BS1}; it obviously implies Theorem~\ref{cm} too.

The notion of quasi-hereditary cover mentioned in Theorem~\ref{qh}(3) was introduced by Rouquier in
\cite[$\S$4.2]{R1}; quasi-hereditary algebras of course go back to the
seminal work of Cline, Parshall and Scott \cite{CPS}. 
The quasi-hereditary cover 
$K_d^{p,q}$ of $H_d^{p,q}$ is Morita
equivalent to another well known quasi-hereditary cover of
$H_d^{p,q}$, namely, 
the (level two)
{\em cyclotomic Schur algebra} 
of Dipper, James and Mathas
from \cite{DJM}; see
also \cite[$\S$6]{AMR} which described the degenerate analogues of
these algebras too.
This Morita equivalence is a consequence of the double centralizer
property. The key point is that our Young modules are the same as the images of the
projective indecomposable modules of the cyclotomic 
Schur algebra under its
Schur functor, as can be proved by an argument involving the special projective
functors $F_i$ analogous to the proof of \cite[Lemma 8.16(ii)]{BS3};
one just needs to know in the cyclotomic Schur algebra setting
that $F_i$ commutes with the Schur functor just like in \cite[Lemma 8.13(iii)]{BS3}.
For a diagrammatic description of this algebra in the spirit of
Khovanov and Lauda, and a
remarkable generalization to other quivers, 
see the recent preprint of Webster \cite{W}.

\section{Khovanov's arc algebra}

Let $\sim$ be the equivalence relation on the set of weight diagrams
defined by $\lambda\sim\mu$ if
$\mu$ is obtained from $\lambda$ by permuting some of the labels
$\down$ and $\up$.
Let $\Lambda$ be any (not necessarily finite) set of weight diagrams
closed under $\sim$.
We are going to recall the definition of 
an algebra
$K_\Lambda$, which is a generalization of Khovanov's arc algebra. Then we will relate
this algebra for particular $\Lambda$ to 
the algebra $K_d^{p,q}$ from the previous section.

We introduced already the notion of a $\lambda$-cap
diagram for any weight diagram $\lambda$.
There is an entirely analogous notion of a {\em $\lambda$-cup
  diagram},
attaching cups $\cup$ and rays down to infinity $|$ below the number
line following the same rules as before. The {\em weight}
$\wt(\tA)$ of a $\lambda$-cup diagram is defined in the same way
as for cap diagrams.
A {\em $\lambda$-circle diagram} means a composite diagram of the form
$\Astacked{\tB}$ obtained by gluing a $\lambda$-cup diagram $\tA$
under
a $\lambda$-cap diagram $\tB$.
Here are two examples (where all vertices not displayed are labelled
$\circ$ or $\cross$):
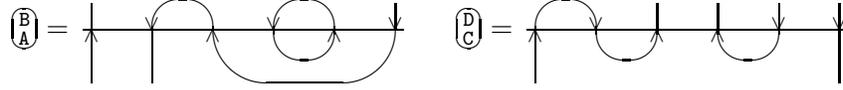
$$
\begin{picture}(265,34)
\put(-30,16){$\Astacked{\tB}$}
\put(-17,17.5){$=$}
\put(138,16){$\Cstacked{\tD}$}
\put(151,17.5){$=$}
\put(-3,20){\line(1,0){121}}
\put(20.1,20.1){$\scriptstyle\down$}
\put(66.1,20.1){$\scriptstyle\down$}
\put(112.1,20.1){$\scriptstyle\down$}
\put(43.1,15.5){$\scriptstyle\up$}
\put(89.1,15.5){$\scriptstyle\up$}
\put(-2.9,15.5){$\scriptstyle\up$}
\put(80.5,20){\oval(23,23)[b]}
\put(80.5,20){\oval(23,23)[t]}
\put(34.5,20){\oval(23,23)[t]}
\put(80.5,20){\oval(69,40)[b]}
\put(0,0){\line(0,1){30}}
\put(23,0){\line(0,1){20}}
\put(115,20){\line(0,1){10}}

\put(165,20){\line(1,0){121}}
\put(179.5,20){\oval(23,23)[t]}
\put(202.5,20){\oval(23,23)[b]}
\put(248.5,20){\oval(23,23)[b]}
\put(168,0){\line(0,1){20}}
\put(283,0){\line(0,1){30}}
\put(260,20){\line(0,1){10}}
\put(237,20){\line(0,1){10}}
\put(214,20){\line(0,1){10}}
\put(188.1,20.1){$\scriptstyle\down$}
\put(257.1,20.1){$\scriptstyle\down$}
\put(280.1,20.1){$\scriptstyle\down$}
\put(211.1,15.5){$\scriptstyle\up$}
\put(234.1,15.5){$\scriptstyle\up$}
\put(165.1,15.5){$\scriptstyle\up$}
\end{picture}
$$

Now we can define the algebra $K_\Lambda$.
As a vector space (over our fixed ground field $F$) $K_\Lambda$ has a distinguished basis consisting of
all the $\lambda$-circle diagrams
 $\Astacked{\tB}$ for all $\lambda\in \Lambda$. 
The multiplication is defined as follows.
Given two basis vectors $\Astacked{\tB}$ and $\Cstacked{\tD}$, their product is zero unless
$\wt(\tB) = \wt(\tC)$.
Assuming $\wt(\tB) = \wt(\tC)$, all the caps and rays in $\tB$ are in
the same positions as the cups and rays in $\tC$. We draw the diagram
$\Astacked{\tB}$ under the diagram $\Cstacked{\tD}$ and stitch
corresponding rays together to obtain a new composite diagram with a
symmetric middle section. 
For example if $\tA,\tB,\tC$ and $\tD$ are as above we get the diagram
$$
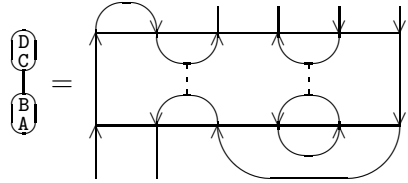
\begin{picture}(120,68)
\put(-32,20){$\Astacked{\tB}$}
\put(-32,44){$\Cstacked{\tD}$}
\put(-27.4,32.1){\line(0,1){8.8}}
\put(-17,33){$=$}
\dashline{2}(80.5,32)(80.5,43)
\dashline{2}(34.5,32)(34.5,43)
\put(0,20){\line(1,0){115}}
\put(0,55){\line(1,0){115}}

\put(20.1,20.1){$\scriptstyle\down$}
\put(66.1,20.1){$\scriptstyle\down$}
\put(112.1,20.1){$\scriptstyle\down$}
\put(43.1,15.5){$\scriptstyle\up$}
\put(89.1,15.5){$\scriptstyle\up$}
\put(-2.9,15.5){$\scriptstyle\up$}

\put(80.5,20){\oval(23,23)[b]}
\put(80.5,20){\oval(23,23)[t]}
\put(34.5,20){\oval(23,23)[t]}
\put(80.5,20){\oval(69,40)[b]}
\put(0,0){\line(0,1){55}}
\put(23,0){\line(0,1){20}}
\put(115,20){\line(0,1){45}}

\put(11.5,55){\oval(23,23)[t]}
\put(34.5,55){\oval(23,23)[b]}
\put(80.5,55){\oval(23,23)[b]}
\put(92,55){\line(0,1){10}}
\put(69,55){\line(0,1){10}}
\put(46,55){\line(0,1){10}}

\put(20.1,55.1){$\scriptstyle\down$}
\put(89.1,55.1){$\scriptstyle\down$}
\put(112.1,55.1){$\scriptstyle\down$}
\put(43.1,50.5){$\scriptstyle\up$}
\put(-2.9,50.5){$\scriptstyle\up$}
\put(66.1,50.5){$\scriptstyle\up$}
\end{picture}
$$
Then we iterate a certain {\em surgery procedure} to be explained 
in the next paragraph
in order to smooth out all
the cup-cap pairs in the symmetric middle section of the diagram
(indicated by dotted lines in the above example). This produces some
new diagrams in which the middle section involves only vertical line
segments. Finally we collapse the middle sections in these new
diagrams to obtain some circle diagrams, and define the desired product to be the sum of the
corresponding basis vectors in $K_\Lambda$.
In the above example applying the left then right
surgeries produces the diagrams
$$
\begin{picture}(260,66)
\put(122,32){then}
\put(-10,20){\line(1,0){115}}
\put(-10,55){\line(1,0){115}}
\put(10.1,20.1){$\scriptstyle\down$}
\put(56.1,20.1){$\scriptstyle\down$}
\put(102.1,20.1){$\scriptstyle\down$}
\put(33.1,15.5){$\scriptstyle\up$}
\put(79.1,15.5){$\scriptstyle\up$}
\put(-12.9,15.5){$\scriptstyle\up$}
\put(70.5,20){\oval(23,23)[b]}
\put(70.5,20){\oval(23,23)[t]}
\put(70.5,20){\oval(69,40)[b]}
\put(-10,0){\line(0,1){55}}
\put(13,0){\line(0,1){55}}
\put(105,20){\line(0,1){45}}
\put(1.5,55){\oval(23,23)[t]}
\put(70.5,55){\oval(23,23)[b]}
\put(82,55){\line(0,1){10}}
\put(59,55){\line(0,1){10}}
\put(36,20){\line(0,1){45}}
\put(10.1,55.1){$\scriptstyle\down$}
\put(79.1,55.1){$\scriptstyle\down$}
\put(102.1,55.1){$\scriptstyle\down$}
\put(33.1,50.5){$\scriptstyle\up$}
\put(-12.9,50.5){$\scriptstyle\up$}
\put(56.1,50.5){$\scriptstyle\up$}
\dashline{2}(70.5,32)(70.5,43)

\put(160,20){\line(1,0){115}}
\put(160,55){\line(1,0){115}}
\put(180.1,20.1){$\scriptstyle\down$}
\put(249.1,20.1){$\scriptstyle\down$}
\put(272.1,20.1){$\scriptstyle\down$}
\put(203.1,15.5){$\scriptstyle\up$}
\put(226.1,15.5){$\scriptstyle\up$}
\put(157.1,15.5){$\scriptstyle\up$}
\put(240.5,20){\oval(23,23)[b]}
\put(240.5,20){\oval(69,40)[b]}
\put(160,0){\line(0,1){55}}
\put(183,0){\line(0,1){55}}
\put(275,20){\line(0,1){45}}
\put(171.5,55){\oval(23,23)[t]}
\put(252,20){\line(0,1){45}}
\put(229,20){\line(0,1){45}}
\put(206,20){\line(0,1){45}}
\put(180.1,55.1){$\scriptstyle\down$}
\put(249.1,55.1){$\scriptstyle\down$}
\put(272.1,55.1){$\scriptstyle\down$}
\put(203.1,50.5){$\scriptstyle\up$}
\put(157.1,50.5){$\scriptstyle\up$}
\put(226.1,50.5){$\scriptstyle\up$}
\end{picture}
$$
Hence we have that
$$
\begin{picture}(125,34)
\put(-48,17){$\phantom{\Astacked{\tB}} \: \phantom{\Cstacked{\tD}}\:\: =$}
\put(-45,15.4){$\Astacked{\tB} \: \Cstacked{\tD} \:\:\phantom{=}$}
\put(0,20){\line(1,0){115}}
\put(11.5,20){\oval(23,23)[t]}
\put(80.5,20){\oval(23,23)[b]}
\put(80.5,20){\oval(69,40)[b]}
\put(23,0){\line(0,1){20}}
\put(0,0){\line(0,1){20}}
\put(115,20){\line(0,1){10}}
\put(92,20){\line(0,1){10}}

\put(69,20){\line(0,1){10}}
\put(46,20){\line(0,1){10}}
\put(20.1,20.1){$\scriptstyle\down$}
\put(89.1,20.1){$\scriptstyle\down$}
\put(112.1,20.1){$\scriptstyle\down$}
\put(43.1,15.5){$\scriptstyle\up$}
\put(66.1,15.5){$\scriptstyle\up$}
\put(-2.9,15.5){$\scriptstyle\up$}
\end{picture}
$$

The surgery procedure is similar to the procedure for computing
$\bar\psi_r$ explained in the last paragraph of the previous section, and goes as follows.
The cup-cap pair to be smoothed either belong to one or two connected
components in the larger diagram. We record a type $1,x,y,1 \otimes 1,
1 \otimes x, 1 \otimes y, x \otimes x, x \otimes y, y \otimes y$
according to whether these one or two components are counter-clockwise
circles ($1$), clockwise circles ($x$) or line segments ($y$).
Then erase the labels on the one or two components and smooth out the
cup-cap pair to get two vertical lines. Finally reintroduce the labels
according to the same rules (\ref{rule1})--(\ref{rule2}) as before with one modification (to take account of a
configuration which did not arise before): in the case $y \otimes y$
if it happens that both the components to start with are lines stretching from infinity
at the bottom to infinity at the top, with one oriented upwards and
the other oriented downwards, then we replace the rule
$y \otimes y \mapsto 0$ with the rule $y \otimes y \mapsto y \otimes y$.
The first surgery in the above example is exactly this situation.

The algebra $K_\Lambda$
has a $\Z_{\geq 0}$-grading defined by declaring that 
the basis vector $\Astacked{\tB}$ is of degree equal to the total number of
clockwise cups and cups in the circle diagram.
The mutually orthogonal idempotents $\{e_\lambda\:|\:\lambda
\in \Lambda\}$ defined by setting $e_\lambda :=
\Astacked{\tB}$ where $\tA$ and $\tB$ are the unique $\lambda$-cup
and $\lambda$-cap diagrams of
weight $\lambda$, respectively, give a basis for the 
degree zero component of $K_\Lambda$.
Hence $K_\Lambda$ is a basic algebra and its degree zero component is 
a (possibly infinite) direct sum of copies of $F$.
We have moreover for arbitrary $\Astacked{\tB}$ that
\begin{equation}
e_\lambda
\begin{picture}(14,0)
\put(2,-2){$\Astacked{\tB}$}
\end{picture} 
= \left\{
\begin{array}{ll}
\Astacked{\tB}&\text{if $\lambda = \wt(\tA)$,}\\
0&\text{otherwise,}
\end{array}
\right.
\qquad
\begin{picture}(14,0)
\put(1,-2){$\Astacked{\tB} $}
\end{picture} 
e_\lambda  = \left\{
\begin{array}{ll}
\Astacked{\tB}&\text{if $\wt(\tB) = \lambda$,}\\
0&\text{otherwise.}
\end{array}
\right.
\end{equation}
Hence $K_\Lambda = \bigoplus_{\lambda,\mu \in \Lambda} e_\lambda
K_\Lambda e_\mu$, so that $K_\Lambda$ is a {\em locally unital
  algebra}.
It is a unital algebra with identity element $1 = \sum_{\lambda \in
  \Lambda} e_\lambda$ if and only if $|\Lambda| < \infty$.
When talking about modules over $K_\Lambda$, we always mean modules $M$
that are locally unital in the sense that $M = \bigoplus_{\la\in\La} e_\lambda
M$.

\begin{Theorem}\label{nw}
Assume that every weight $\lambda \in \Lambda$ either has finitely many vertices labelled $\down$ or finitely many vertices labelled $\up$.
Then the category $K_\La\operatorname{-MOD}$ of finite dimensional
{\em graded} left $K_\La$-modules is a graded highest weight category with
projective
indecomposable modules $P(\lambda) := K_\La e_\la$,
standard modules $V(\lambda)$
and irreducible modules $L(\lambda)$; the gradings on these modules are fixed so that $L(\lambda)$ is
one-dimensional concentrated in degree $0$, and the canonical
homomorphisms
$P(\lambda)\twoheadrightarrow V(\lambda) \twoheadrightarrow
L(\lambda)$ are grading-preserving.
Moreover:
\begin{itemize}
\item[(1)]
For $\la,\mu \in \La$, 
the graded decomposition number
$[V(\lambda):L(\mu)]_q$ is equal to
$q^n$ if $\lambda \supset \mu$, where $n$ is the number of clockwise caps
in the unique $\lambda$-cap diagram of weight $\mu$;
otherwise, $[V(\lambda):L(\mu)]_q = 0$.
\item[(2)] The positively graded 
algebra $K_\Lambda$ is standard Koszul, i.e. the irreducible modules
$L(\lambda)$
and the standard modules $V(\lambda)$
have linear projective resolutions.
Moreover the associated
Kazhdan-Lusztig polynomials
$$
p_{\lambda,\mu}(q) := \sum_{i \geq 0} q^i \dim
\operatorname{Ext}_{K_\Lambda}^i(V(\lambda),L(\mu))
$$
are given explicitly by the following recurrence.
First 
$p_{\lambda,\mu}(q) = 0$ unless $\lambda \leq \mu$,
and $p_{\lambda,\lambda}(q) = 1$.
Now assume that $\lambda < \mu$. Pick $i < j$ such that the $i$th vertex
  of $\lambda$ is labelled $\down$, the $j$th vertex is labelled
  $\up$, and all vertices in between are labelled $\circ$ or
  $\cross$. For any weight diagram $\nu$ and $x,y \in \{\circ,\cross,\down,\up\}$
let $\nu\scriptstyle{[}xy\scriptstyle{]}$ be the weight diagram obtained from $\nu$ by relabelling vertex
$i$ by $x$ and vertex $j$ by $y$.
Then
$$
p_{\lambda,\mu}(q) = 
\left\{
\begin{array}{ll}
p_{\lambda\scriptscriptstyle{[}\circ\circ\scriptscriptstyle{]},\mu\scriptscriptstyle{[}\circ\circ\scriptscriptstyle{]}}(q)+q
p_{\lambda\scriptscriptstyle{[\wedge\vee]},\mu}(q)&\text{if $\mu = \mu\scriptstyle{[}\down\up\scriptstyle{]}$,}\\
qp_{\lambda\scriptscriptstyle{[\wedge\vee]},\mu}(q)&\text{otherwise.}
\end{array}
\right.
$$
\item[(3)]
For fixed $\mu \in \Lambda$, 
we have that $p_{\lambda,\mu}(1) \leq 1$ for all $\lambda
\in \Lambda$ if and only if it is impossible to find vertices $i < j <
k < l$ whose labels in $\mu$ are $\up,\down,\up,\down$, respectively.
In that case $L(\mu)$ possesses a BGG-type resolution 
$$
\cdots \rightarrow V_1(\mu) \rightarrow V_0(\mu)
\rightarrow L(\mu) \rightarrow 0
$$ 
with $V_i(\mu)
=\bigoplus_{\lambda \text{\:s.t.\:} p_{\lambda,\mu}(q)=q^i} V(\lambda) \langle i \rangle$.
\item[(4)] For $\lambda \in \Lambda$, we have that
$P(\lambda)$ is injective if and only if $\lambda$ is restricted.
\item[(5)] The algebra $K_\Lambda$ decomposes into blocks as 
$K_\Lambda = \bigoplus_{\Gamma \in \Lambda /
    \sim} K_\Gamma$.
\end{itemize}
\end{Theorem}

The precise connection between $K_\Lambda$ and
the algebra in Theorem~\ref{qh} is explained by the next theorem.

\begin{Theorem}\label{id}
The endomorphism algebra $K_d^{p,q}
= \End_{H_d^{p,q}}(\bigoplus_{\lambda \Vdash d} Y(\lambda))^{\op}$ from Theorem~\ref{qh} is 
canonically isomorphic to $K_\Lambda$
for $\Lambda := \{\lambda \Vdash d\}$ (interpreting
bipartitions as weight diagrams as explained in the previous section).
Under the isomorphism, $\Astacked{\tB} \in K_\Lambda$
corresponds to a
map
$Y(\lambda) \rightarrow Y(\mu)$ where
$\lambda := \wt(\tA)$ and $\mu := \wt(\tB)$.
This map is
defined
on $v_\T^\lambda \in Y(\lambda)$ by
drawing the diagram $\Tstacked{\lambda}$ (as defined in the previous section) under the diagram
$\Astacked{\tB}$
then iterating the surgery procedure in exactly the same way as in the
definition of
the
multiplication of $K_\Lambda$, to obtain 
some diagrams $\Sstacked{\mu}$
hence a sum of basis
vectors $v_\Stab^\mu \in Y(\mu)$.
\end{Theorem}

\vspace{2mm}
\noindent
{\em Notes.}
Assuming $|\Lambda| < \infty$, the algebra $K_\Lambda$ is the
{quasi-hereditary cover of the generalized Khovanov algebra} $H_\La$ from
\cite{BS1}, which was
introduced
already in \cite{CK} and
\cite[$\S$5]{S}. 
More precisely, letting $\Lambda^\circ := \{\lambda \in
\Lambda\:|\:\lambda\text{ is restricted}\}$, 
the {\em generalized Khovanov algebra}
is the (symmetric) subalgebra $$
H_\La := \bigoplus_{\lambda,\mu
  \in \Lambda^\circ} e_\lambda K_\Lambda e_\mu
$$ 
of $K_\Lambda$, and
there is a double centralizer property implying that $K_\Lambda\operatorname{-mod}$ is a
highest weight cover of $H_\Lambda\operatorname{-mod}$ in the sense of
\cite[$\S$4.2]{R1}; see \cite[$\S$6]{BS2}.
In the special case that the weights in $\Lambda$ have the same number
of labels $\up$ as $\down$,
the algebra $H_\Lambda$ is exactly the original arc algebra introduced
by Khovanov in \cite{K2}; in that case the diagrams indexing the basis for
$H_\Lambda$ involve only (closed) circles, no line segments, and 
the multiplication has an elegant formulation in terms of a certain
TQFT. This interpretation is the key to
proving that the multiplication as formulated above is well defined independent of the
order of the surgery procedures and that it is associative; see \cite{BS1}.
Also in \cite{BS1} we showed that
the diagram bases for both $K_\Lambda$ and for $H_\Lambda$ are cellular bases in the sense of \cite{GL}; in fact they are examples of {\em graded cellular algebras}
as recently formalized by Hu and Mathas \cite[$\S$2]{HM}.

In the statement of Theorem~\ref{nw}, we have used the language of highest weight categories from \cite{CPS} rather 
than of quasi-hereditary algebras because $K_\Lambda$ is not necessarily finite
dimensional. The assumption on $\Lambda$ 
in the opening sentence of the theorem is necessary since without it the analogues of the standard modules $V(\lambda)$ have infinite length, but
the remaining statements (1)--(5) of the theorem remain 
true without this assumption.
Theorem~\ref{nw}(1) 
is \cite[Theorem 5.2]{BS1}
and (5) is an easy consequence;
for (2), (3) and (4) 
see \cite[$\S$5, $\S$7 and $\S$6]{BS2}, respectively.
The recurrence relation for Kazhdan-Lusztig polynomials in
Theorem~\ref{nw}(2) is the same as the recurrence for
the Kazhdan-Lusztig polynomials attached to Grassmannians 
discovered by Lascoux and Sch\"utzenberger \cite[Lemme 6.6]{LS}. This
coincidence is explained by Theorem~\ref{cato} below.
There is also
a closed formula for these Kazhdan-Lusztig polynomials due again to Lascoux
and Sch\"utzenberger; see
\cite[(5.3)]{BS2} for an equivalent formulation in terms of cap
diagrams.

Theorem~\ref{id} 
is \cite[Corollary 8.15]{BS3}.
One consequence is that the level two Hecke algebra
$H^{p,q}_d$ is itself Morita equivalent to the generalized Khovanov
algebra $H_\Lambda$ for $\Lambda := \{\lambda \Vdash d\}$; see
\cite[Theorem 6.2]{BS3}.

\section{Category $\mathcal O$ for Grassmannians}

Let $\mathfrak{g} := \mathfrak{gl}_{m+n}(\C)$, 
$\mathfrak{t}$ be the Cartan subalgebra of diagonal
matrices, and $\mathfrak{b}$ be the Borel subalgebra of upper
triangular matrices. Let $\eps_1,\dots,\eps_{m+n}$ be the basis for
$\mathfrak{t}^*$ dual to the obvious basis of $\mathfrak{t}$
consisting of the diagonal matrix units, and let $(.,.)$ be the bilinear form on
$\mathfrak{t}^*$
with respect to which the $\eps_i$'s are orthonormal.
For each $\lambda \in \mathfrak{t}^*$,
let $L(\lambda)$ be an
irreducible $\mathfrak{g}$-module of $\mathfrak{b}$-highest weight
$\lambda$.
Finally from now until the end of the section we let
\begin{equation}\label{ol}
\Lambda := \left\{\lambda 
\in \mathfrak{t}^*\:\Bigg|\:
\begin{array}{ll}
(\lambda+\rho,\eps_1),\dots,(\lambda+\rho,\eps_{m+n}) \in \Z\\
(\lambda+\rho,\eps_1) > \cdots > (\lambda+\rho,\eps_m)\\
(\lambda+\rho,\eps_{m+1}) > \cdots > (\lambda+\rho,\eps_{m+n})
\end{array}
\right\}
\end{equation}
 where
$\rho := -\eps_2-2\eps_3-\cdots-(m+n-1)\eps_{m+n}$.

We are interested in the category 
$\mathcal O(m,n)$ of all 
$\mathfrak{g}$-modules that are semisimple over $\mathfrak{t}$
and possess a composition series with composition factors of
the form $L(\lambda)$ for $\lambda \in \Lambda$.
This is the sum of all ``integral'' blocks of the parabolic analogue
of the usual Bernstein-Gelfand-Gelfand category $\mathcal O$
corresponding to the standard parabolic subalgebra $\mathfrak{p}$
with Levi factor $\mathfrak{gl}_m(\C)\oplus\mathfrak{gl}_n(\C)$.
It is a highest weight category with
irreducible modules $\{L(\lambda) \:|\:\lambda \in \Lambda\}$, standard modules
$\{V(\lambda) \:|\:\lambda \in \Lambda\}$ (which can be constructed
explicitly as parabolic Verma modules) and projective indecomposable modules
$\{P(\lambda) \:|\:\lambda \in \Lambda\}$.

We identify $\lambda \in \Lambda$ with a weight
diagram in the sense of section 3 by putting the symbol $\down$
at all vertices indexed by the set $\{(\lambda+\rho,\eps_1),\dots,(\lambda+\rho,\eps_m)\}$ and
the symbol $\up$ at all vertices indexed by the set
$\{(\lambda+\rho,\eps_{m+1}),\dots,(\lambda+\rho,\eps_{m+n})\}$,
interpreting both labels as $\cross$ and neither as $\circ$ as before.
These weight diagrams are slightly different from the weight diagrams
arising from bipartitions in section 3: there are now infinitely many
$\circ$'s both to the left and the right.
Viewing $\La$ as a set of weight diagrams in this way, it is closed under $\sim$.
Let $K(m,n)$ be the arc algebra $K_\Lambda$ from the previous section
for this choice of $\La$.

\begin{Theorem}\label{cato}
Let $P := \bigoplus_{\lambda \in \Lambda} P(\lambda)$.
The locally finite
endomorphism\footnote{A {\em locally finite} endomorphism means one that 
is zero on all but finitely many $P(\lambda)$'s.} algebra
$\End_{\mathfrak{g}}^{fin}(P)^{\operatorname{op}}$
is 
isomorphic to the arc algebra $K(m,n)$ so
that 
$e_\lambda \in K(m,n)$ corresponds to the projection onto
the summand $P(\lambda)$.
Fixing such an isomorphism, the functor 
\begin{equation}\label{ark}
\hom_{\mathfrak{g}}(P, ?):\mathcal{O}(m,n)\rightarrow
K(m,n)\operatorname{-mod}
\end{equation}
is an equivalence of categories 
sending $P(\lambda), V(\lambda),L(\lambda) \in \mathcal{O}(m,n)$ 
to (the ungraded
versions of) the $K(m,n)$-modules with the same name from Theorem~\ref{nw}.
\end{Theorem}

The main idea for the proof of
Theorem~\ref{cato} is to 
exploit another Schur-Weyl duality which relates $\mathcal O(m,n)$ to the
level two Hecke algebras
$H_d^{p,q}$ from (\ref{ker}).
To formulate this, we fix integers $p,q \in \Z$ so that $p-m=q-n$.
Suppose $\lambda \Vdash d$ is a bipartition such that
$h(\lambda^L) \leq m$ and $h(\lambda^R) \leq n$, where $h(\mu)$ denotes
the 
{\em height} (number of non-zero parts) of a partition $\mu$.
Let $\bar{\lambda} \in \Lambda$ be the weight obtained by viewing
$\lambda$ as a weight diagram as in section 3, then changing
the labels of all the
vertices indexed by integers $\leq (p-m)$ from $\cross$ to $\circ$. For example, if $p \leq q$ then the empty
bipartition $\varnothing$ becomes the diagram
$$
\begin{picture}(200,22)
\put(-48,.25){$\bar{\varnothing} =
\:\:\,\cdots$}
\put(2,3){\line(1,0){223}}
\put(192.5,.3){$\circ$}
\put(171,-8.3){$_q$}
\put(101.5,13.3){$_p$}
\put(26.5,13.3){$_{p-m}$}
\put(28,-8.3){$_{q-n}$}
\put(215.5,.3){$\circ$}
\put(124.3,-1.6){$\up$}
\put(147.3,-1.6){$\up$}
\put(170.3,-1.6){$\up$}
\put(54.3,1){$\cross$}
\put(77.3,1){$\cross$}
\put(100.3,1){$\cross$}
\put(31.5,.3){$\circ$}
\put(8.5,.3){$\circ$}
\put(233,.5){$\cdots$}
\end{picture}
$$
\vspace{0.1mm}

\noindent
A key point is that $L(\bar{\varnothing})$ is an irreducible projective
module in $\mathcal O(m,n)$. Hence 
for $d \geq 0$ the module
$L(\bar{\varnothing}) \otimes V^{\otimes d}$ is projective
in $\mathcal O(m,n)$ too, where $V$ is the natural
$\mathfrak{g}$-module of column vectors.

\begin{Theorem}\label{another}
The algebra $H_d^{p,q}$ acts on the right on 
$L(\bar{\varnothing}) \otimes V^{\otimes d}$ so that
$s_r$ flips the $r$th and $(r+1)$th tensors in 
$V^{\otimes d}$ as usual, and $L_1$
acts as the endomorphism $\sum_{i,j=1}^{m+n} e_{i,j} \otimes e_{j,i} \otimes 1 \otimes\cdots\otimes 1$ (where $e_{i,j}$ denotes the $ij$-matrix unit in
$\mathfrak{g}$).
This action induces a {\em surjective} homomorphism
\begin{equation}
H_d^{p,q} \twoheadrightarrow \End_{\mathfrak{g}}(L(\bar{\varnothing}) \otimes V^{\otimes d})^{\op},
\end{equation}
which is
an isomorphism if and only if $d \leq \min(m,n)$.
Moreover the exact functor
\begin{equation}\label{thef}
\hom_{\mathfrak{g}}(L(\bar{\varnothing}) \otimes V^{\otimes d}, ?):\mathcal O(m,n) \rightarrow
H_d^{p,q}\operatorname{-mod}
\end{equation}
sends $P(\bar{\lambda})$ to $Y(\lambda)$ 
for all $\lambda \Vdash d$ with $h(\lambda^L) \leq m,
  h(\lambda^R) \leq n$, and it is fully faithful on the additive subcategory
of $\mathcal O(m,n)$ generated by these projective modules.
\end{Theorem}

To explain how to deduce Theorem~\ref{cato} 
from Theorem~\ref{another}, let $Y$, $H_d^{p,q}$ and $K_d^{p,q}$ be as in Theorem~\ref{qh}, and $K(m,n)$ be as in Theorem~\ref{cato}. Set
$P := \bigoplus_\la P(\bar{\lambda})$, 
$e := \sum_\la e_\lambda \in K_d^{p,q}$
and $\bar{e} := \sum_\lambda e_{\bar{\lambda}} \in K(m,n)$, all sums over
$\lambda \Vdash d$ such that $h(\lambda^L) \leq m, h(\lambda^R) \leq
n$.
It is obvious from Theorem~\ref{id} and the diagrammatic 
definition of the algebra $K(m,n)$
that $e K_d^{p,q} e \cong \bar e K(m,n) \bar e$.
Applying Theorem~\ref{another}
we get that
\begin{equation}\label{oldmethod}
\End_{\mathfrak{g}}(P)^{\op}
\cong 
\End_{H^{p,q}_d}\left( Y e\right)^{\op}
\cong  e K_d^{p,q} e
\cong \bar{e} K(m,n) \bar{e}.
\end{equation}
Theorem~\ref{cato} follows from this 
on
observing given any $\sim$-equivalence class
$\Gamma$ of weights 
from $\Lambda$ 
that we can choose $p,q$ and $d$ so that 
all weights in $\Gamma$ are of the form 
$\bar\lambda$ for $\lambda \Vdash d$ with $h(\lambda^L) \leq m,
h(\lambda^R)\leq n$.

\vspace{2mm}
\noindent
{\em Notes.}
For a detailed account of the general theory of parabolic category $\mathcal O$ for a
semisimple Lie algebra, see \cite[ch. 9]{Hbook}.

Theorem~\ref{another} is proved in \cite{BKschur}, \cite{BKariki}, and the
deduction of
Theorem~\ref{cato} following the argument just sketched can be found
in detail in \cite[$\S$8]{BS3}.
For the special case $m=n$, the identification of 
the principal block of $\mathcal{O}(m,n)$ with the 
corresponding block of the diagram algebra $K(m,n)$ was established
earlier by Stroppel \cite[Theorem 5.8.1]{S} using an explicit
presentation for the endomorphism algebra of a minimal projective
generator for the category of perverse sheaves on the Grassmannian
found by Braden in \cite{Br} (this category of perverse sheaves being
equivalent to the principal block of $\mathcal{O}(m,n)$ thanks to the
Beilinson-Bernstein localization theorem). The idea that there should
be such 
an isomorphism originates in unpublished work of Braden and Khovanov.

With Theorem~\ref{cato} in hand, 
all the statements of Theorem~\ref{nw} for $\Lambda$ as in (\ref{ol})
are equivalent to previously known facts about $\mathcal O(m,n)$.
In particular, the fact that blocks of $\mathcal O(m,n)$ are
Koszul coming from Theorem~\ref{nw}(2) is a very
special case of the general results of Beilinson, Ginzburg and Soergel
\cite{BGS} and Backelin \cite{Back}; the present approach is more
explicit and
algebraic in nature.
 The possibility of computing the composition multiplicities of
parabolic Verma modules 
in a geometry-free way
as in Theorem~\ref{nw}(1) 
was first realized by Enright
and Shelton in \cite{ES}. 

For an application of Theorem~\ref{cato} to classify the indecomposable
projective functors on $\mathcal O(m,n)$ in the sense of \cite{BG}, see \cite[Theorem
1.2]{BS3}.
The problem of classifying
indecomposable projective functors on parabolic category $\mathcal O$
for an arbitrary parabolic of an arbitrary semisimple Lie algebra
remains open in general.

\section{The general linear supergroup}

In this section we discuss the application to the representation
theory of the general linear supergroup $G := GL(m|n)$ over the ground field
$\C$. Using scheme-theoretic language, $G$ can be regarded as a functor
from the category of commutative superalgebras over $\C$
to the category of groups,
mapping a commutative superalgebra $A = A_{\0}\oplus A_{\1}$
to the group $G(A)$ of all  invertible $(m+n) \times (m+n)$
matrices of the form
\begin{equation}\label{supermat}
g=
\left(
\begin{array}{l|l}
a&b\\\hline
c&d
\end{array}
\right)
\end{equation}
where $a$ (resp. $d$) is an $m \times m$ (resp. $n \times n$)
matrix with entries in
$A_{\0}$, and $b$ (resp. $c$) is an $m\times n$ (resp. $n \times m$)
matrix with entries in
$A_{\1}$.
Let $B$ and $T$ be the standard choices of Borel subgroup\label{borels}
and maximal torus: for each commutative superalgebra $A$, the
groups
$B(A)$ and $T(A)$ consist of all
matrices $g \in G(A)$ that are upper triangular and diagonal, respectively.
Let $\eps_1,\dots,\eps_{m+n}$ be the usual basis for the
character group $X(T)$ of $T$, i.e. $\eps_r$ picks
out the $r$th diagonal entry of a diagonal matrix.
Equip $X(T)$ with a symmetric bilinear form $(.,.)$
such that $(\eps_r,\eps_s) = (-1)^{\bar r} \delta_{r,s}$, where
$\bar r := \0$ if $1 \leq r \leq m$ and $\bar r := \1$ if $m+1 \leq r \leq m+n$.
 Let
\begin{equation}
\label{XT}
\Lambda
:= \left\{
\la \in X(T)\:\bigg|\:
\begin{array}{c}
\:\:\,(\la+\rho,\eps_1) > \cdots > (\la+\rho,\eps_m),\\
(\la+\rho,\eps_{m+1}) < \cdots < (\la+\rho,\eps_{m+n})
\end{array}
\right\}
\end{equation}
denote the set of {\em dominant weights}, where
$\rho := -\eps_2-2\eps_3-\cdots-(m-1)\eps_m+(m-1) \eps_{m+1}+(m-2)\eps_{m+2}+\cdots+(m-n)\eps_{m+n}$.

We are interested here in the abelian category $\Rep{G}$ of finite dimensional
representations of $G$; 
we allow arbitrary (not necessarily even) morphisms between
$G$-modules so that the existence of
kernels and cokernels is not quite obvious.
The category $\Rep{G}$ is a highest weight category with irreducible
objects
$\{L(\lambda)\:|\:\lambda \in \Lambda\}$,
standard objects
$\{V(\lambda)\:|\:\lambda \in \Lambda\}$ and projective
indecomposables
$\{P(\lambda)\:|\:\lambda \in \Lambda\}$. In this setting the standard
objects are called {\em Kac modules} and they can be constructed by
geometric induction from $B$: we have that $V(\lambda) = H^0(G / B,
\mathscr L(\lambda)^*)^*$ where $G / B$ is Manin's flag superscheme
and $\mathscr L(\lambda)$ denotes the $G$-equivariant line bundle on
$G / B$ attached to the weight $\lambda$.

We identify $\lambda \in \Lambda$ with a weight
diagram obtained by putting the symbol $\down$ on vertices indexed by
the set
$\{(\lambda+\rho,\eps_1),\dots,(\lambda+\rho,\eps_m)\}$ and the symbol
$\up$ on all vertices
indexed by the set $\Z \setminus \{(\lambda+\rho,\eps_{m+1}),\dots,(\lambda+\rho,\eps_{m+n})\}$,
writing $\cross$ for both and $\circ$ for neither as usual.
Unlike the situations considered in sections 3 and 5, the
non-trivial $\sim$-equivalence classes in $\Lambda$ are all infinite,
and all but finitely many vertices\footnote{The reader concerned by the apparent lack of symmetry here 
should note that we have already made a choice earlier in defining the
parities $\bar r\:\:(1 \leq r \leq m+n)$. We could also have set
$\bar r := \1$ for $1 \leq r \leq m$ and $\bar r := \0$ for $m+1 \leq r \leq m+n$,
a path which leads to weight diagrams
in which all but finitely many vertices are labelled $\down$.}
 are labelled $\up$.
Let $K(m|n)$ be the arc algebra $K_\La$ from section 4 for this new choice of $\La$.

\begin{Theorem}\label{miss}
Let $P := \bigoplus_{\lambda \in \Lambda} P(\lambda)$. The locally
finite endomorphism algebra $\End^{fin}_G(P)^{\op}$ is isomorphic to 
the arc algebra $K(m|n)$ so that $e_\lambda \in K(m|n)$
corresponds to the projection onto $P(\lambda)$. Fixing such an isomorphism, the functor
\begin{equation}\label{bere}
\hom_G(P, ?):\Rep{G} \rightarrow K(m|n)\operatorname{-mod}
\end{equation}
is an
equivalence of categories sending $P(\lambda), V(\lambda), L(\lambda) 
\in \Rep{G}$
to the $K(m|n)$-modules with the same name.
\end{Theorem}

Again the proof involves a Schur-Weyl duality, though it is a bit more
subtle than in the previous section due to the existence of infinite
$\sim$-equivalence classes of weights in $\Lambda$.
To formulate the key result, we fix integers $p \leq q$.
Suppose $\lambda \Vdash d$ is a bipartition such that
$h(\lambda^L) \leq m, w(\lambda^L) \leq
n+q-p, h(\lambda^R) \leq m+q-p, w(\lambda^R) \leq n$, where $h(\mu)$ denotes height and
$w(\mu)$ denotes
the 
{\em width} (largest part) of a partition $\mu$.
Let $\hat{\lambda} \in \Lambda$ be the weight obtained by viewing
$\lambda$ as a weight diagram as in section 3, then changing the labels of
all the
vertices indexed by integers $\leq (p-m)$ 
from $\cross$ to $\up$ and all the ones indexed by integers $> (q+n)$
from $\circ$ to $\up$.
For example the empty
bipartition $\varnothing$ becomes
$$
\begin{picture}(-365,25)
\put(-222,14){$_p$}
\put(-153.3,14){$_q$}
\put(-294,14){$_{p-m}$}
\put(-90,14){$_{q+n}$}
\put(-363,-1.5){$\hat{\varnothing} \:=\:\:\cdots$}
\put(-22,-1.5){$\cdots$}
\put(-313,1.2){\line(1,0){283}}
\put(-288.7,-3.3){$\scriptstyle\up$}
\put(-311.7,-3.3){$\scriptstyle\up$}
\put(-153.7,-3.3){$\scriptstyle\up$}
\put(-176.7,-3.3){$\scriptstyle\up$}
\put(-199.7,-3.3){$\scriptstyle\up$}
\put(-269.2,-0.6){$\scriptstyle\times$}
\put(-246.2,-0.6){$\scriptstyle\times$}
\put(-223.2,-0.6){$\scriptstyle\times$}
\put(-131.4,-1.4){$\circ$}
\put(-108.4,-1.4){$\circ$}
\put(-85.4,-1.4){$\circ$}
\put(-61.7,-3.3){$\scriptstyle\up$}
\put(-38.7,-3.3){$\scriptstyle\up$}
\end{picture}
$$
\vspace{0.1mm}

\noindent
Again $L(\hat{\varnothing})$ is an irreducible projective
module in $\Rep{G}$. Hence 
for $d \geq 0$ the module
$L(\hat{\varnothing}) \otimes V^{\otimes d}$ is projective
in $\Rep{G}$ too, where $V$ is the natural
$G$-module of column vectors
with standard basis $v_1,\dots,v_m,v_{m+1},\dots,v_{m+n}$
and $\Z_2$-grading defined by putting $v_r$ in degree 
$\bar r$.

\begin{Theorem}\label{another2}
The algebra $H_d^{p,q}$ acts on the right on 
$L(\hat{\varnothing}) \otimes V^{\otimes d}$ so that
$s_r$ flips the $r$th and $(r+1)$th tensors in 
$V^{\otimes d}$ with a sign if both vectors are odd, and $L_1$
acts as the endomorphism $\sum_{i,j=1}^{m+n} (-1)^{\bar j}e_{i,j} \otimes e_{j,i} \otimes 1 \otimes\cdots\otimes 1$ (where $e_{i,j}$ denotes the $ij$-matrix unit in
the Lie superalgebra of $G$).
This action induces a {\em surjective} homomorphism
\begin{equation}\label{newiso}
H_d^{p,q} \twoheadrightarrow \End_G(L(\hat{\varnothing}) \otimes V^{\otimes d})^{\op},
\end{equation}
which is
an isomorphism if and only if $d \leq \min(m,n) + q-p$.
Moreover the exact functor
\begin{equation}\label{theeef}
\hom_{\mathfrak{g}}(L(\hat{\varnothing}) \otimes V^{\otimes d}, ?):
\Rep{G} \rightarrow
H_d^{p,q}\operatorname{-mod}
\end{equation}
sends $P(\hat{\lambda})$ to $Y(\lambda)$ 
for all {\em restricted} $\lambda \Vdash d$ with $h(\lambda^L) \leq m,
w(\lambda^L) \leq n+q-p,
h(\lambda^R) \leq m+q-p,
  w(\lambda^R) \leq n$, and it is fully faithful on the additive subcategory
of $\Rep{G}$ generated by these projective modules.
\end{Theorem}

If we mimic (\ref{oldmethod})
with $P := \bigoplus_\la P(\hat{\lambda})$, 
$e := \sum_\la e_\lambda \in K_d^{p,q}$
 and $\hat e := \sum_\la e_{\hat\lambda} \in K(m|n)$, all sums over restricted
$\lambda \Vdash d$ such that $h(\lambda^L) \leq m, w(\lambda^L) \leq
n+q-p,
h(\lambda^R) \leq m+q-p,w(\lambda^R) \leq n$,
we get that
\begin{equation}\label{oldmethod2}
\End_{G}(P)^{\op}
\cong
\End_{H_d^{p,q}}(Y e)^{\op}
\cong 
 e K_d^{p,q} e
\cong \hat e K(m|n) \hat e.
\end{equation}
Given any {\em finite} set $\Gamma$ of weights 
from the same $\sim$-equivalence class in $\Lambda$,
there exist $p \leq q$ and $d$ such that all the weights
in $\Gamma$ are of the form $\hat{\lambda}$ for restricted $\lambda \Vdash d$
with $h(\lambda^L)\leq m,w(\lambda^L) \leq n+q-p,h(\lambda^R)\leq m+q-p,w(\lambda^R) \leq n$. So the endomorphism
algebra of $\bigoplus_{\lambda \in \Gamma} P(\lambda)$ can be worked
out from (\ref{oldmethod2}). This should at least make
Theorem~\ref{miss} rather plausible although this argument is no
longer quite a proof.

\vspace{2mm}
\noindent
{\em Notes.}
The curious observation that $\Rep{G}$ (with not necessarily homogeneous morphisms) 
is abelian is made in \cite[$\S$2.5]{CL}. In \cite{BS4} we 
worked instead in a certain full subcategory
$\mathscr F(m|n)$ of $\Rep{G}$ which is obviously abelian.
Since every $M \in \Rep{G}$ is isomorphic via a not necessarily homogeneous isomorphism to an object in $\mathscr F(m|n)$ it follows that 
$\Rep{G}$ is abelian too.
The fact that $\Rep{G}$ is a highest weight category is established in 
\cite[Theorem 4.47]{B}.

Theorem~\ref{miss} is proved in \cite[Theorem 1.1]{BS4} (see also
\cite[Lemmas 5.8--5.9]{BS4}) by carefully
taking a limit as $p \rightarrow -\infty$ and $q \rightarrow \infty$.
Combined also with Theorem~\ref{nw}, it has several consequences for the
structure of $\Rep{G}$. In particular using Theorem~\ref{nw}(2), we
get that the category $\Rep{G}$ possesses a hidden {Koszul grading}
in the spirit of \cite{BGS}.

From Theorem~\ref{nw}(1), we recover
the following formula proved originally in \cite{B} for the composition multiplicities of Kac
modules:
\begin{equation}\label{dec}
[V(\lambda):L(\mu)] = \left\{
\begin{array}{ll}
1&\text{if $\lambda \supset \mu$,}\\
0&\text{otherwise.}
\end{array}
\right.
\end{equation}
In this setting, the Kazhdan-Lusztig polynomials from
Theorem~\ref{nw}(2) were introduced originally by Serganova, motivated
by the observation that
\begin{equation}\label{chform}
\operatorname{ch} L(\mu) = \sum_{\lambda\leq\mu} p_{\lambda,\mu}(-1)
\operatorname{ch} K(\lambda).
\end{equation}
Using geometric induction techniques, Serganova 
computed the characters of the finite dimensional irreducible
$G$-modules
already in \cite{Se}.
Recently Musson and Serganova \cite{MS} have explained the connection
between the alternating sum formula for the composition multiplicities
$[V(\lambda):L(\mu)]$ arising from Serganova's original work and the formula
(\ref{dec}) from \cite{B} in purely combinatorial terms.

For atypical $\lambda$, the character formula (\ref{chform}) involves an infinite sum so
does not obviously imply a dimension formula for the irreducible $G$-modules,
but Su and Zhang were able to make some simplifications to
deduce such a result; see \cite{SZ}.
Theorem~\ref{nw}(3) gives a BGG-type resolution for certain irreducible
$G$-modules, including all {\em polynomial representations} for which this 
was established already by Cheng, Kwon and
Lam \cite{CKL}.
Finally we mention that on combining Theorems~\ref{cato} and \ref{miss}, one
can prove the ``super duality'' conjecture of Cheng, Wang and Zhang
\cite{CWZ}; a more direct proof was found subsequently in \cite{CL}.

Theorem~\ref{another2} is established in \cite[$\S$3]{BS4}. The
precise bound on $d$ for (\ref{newiso}) to be an isomorphism 
is a new observation; it follows from dimension considerations
similar to \cite[Theorem 3.9]{BS4}.
The coincidence between the composition multiplicities of Kac modules
from (\ref{dec}) and of Specht modules from Theorem~\ref{cm} was first
pointed out by Leclerc and Miyachi in \cite{LM}, and is nicely explained
by the Schur-Weyl duality in 
Theorem~\ref{another2}.

\ifbrauer@
\section{The walled Brauer algebra}

In this section we formulate an even more recent result
which explains some striking combinatorial coincidences observed recently by Cox and De Visscher.
These coincidences suggest a functorial link
between representations of the walled Brauer algebra $B_{r,s}(\delta)$ and of the generalised Khovanov algebra $K_\Lambda$ in a situation in which weights in $\Lambda$ have infinitely many vertices labelled $\up$ and infinitely many vertices labelled $\down$, the one situation in which we did not know of an occurence of $K_\Lambda$ ``in nature'' before.

Fix a parameter $\delta \in \C$. The walled Brauer algebra $B_{r,s}(\delta)$ is a certain subalgebra of the classical Brauer algebra $B_{r+s}(\delta)$. As a $\C$-vector space it has
dimension $(r+s)!$, with
a basis consisting of isotopy classes of diagrams
drawn in a rectangle with $(r+s)$ vertices on its top and bottom edges,
and a vertical wall separating the leftmost $r$ from the rightmost $s$ vertices.
Each vertex must be connected
to exactly one other vertex by a smooth curve drawn in the interior of rectangle, connected pairs of vertices on opposite edges
must lie on the same side of the wall, and connected pairs of vertices on the same edge must lie on opposite sides of the wall.
For example here are two 
basis vectors in $B_{2,2}(\delta)$:
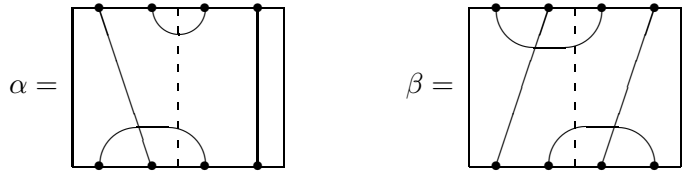
$$
\begin{picture}(150,72)
\put(-24,34){$\alpha=$}
\put(8,4.2){$\scriptstyle\bullet$}
\put(28,4.2){$\scriptstyle\bullet$}
\put(48,4.2){$\scriptstyle\bullet$}
\put(68,4.2){$\scriptstyle\bullet$}
\put(8,64.2){$\scriptstyle\bullet$}
\put(28,64.2){$\scriptstyle\bullet$}
\put(48,64.2){$\scriptstyle\bullet$}
\put(68,64.2){$\scriptstyle\bullet$}
\put(0,6){\line(1,0){80}}
\put(0,6){\line(0,1){60}}
\put(0,66){\line(1,0){80}}
\put(80,66){\line(0,-1){60}}
\dashline{3}(40,6)(40,66)
\put(30,6){\line(-1,3){20}}
\put(70,6){\line(0,1){60}}
\put(30.3,6){\oval(40,30)[t]}
\put(40.3,66){\oval(20,20)[b]}
\end{picture}
\begin{picture}(80,72)
\put(-24,34){$\beta=$}
\put(8,4.2){$\scriptstyle\bullet$}
\put(28,4.2){$\scriptstyle\bullet$}
\put(48,4.2){$\scriptstyle\bullet$}
\put(68,4.2){$\scriptstyle\bullet$}
\put(8,64.2){$\scriptstyle\bullet$}
\put(28,64.2){$\scriptstyle\bullet$}
\put(48,64.2){$\scriptstyle\bullet$}
\put(68,64.2){$\scriptstyle\bullet$}
\put(0,6){\line(1,0){80}}
\put(0,6){\line(0,1){60}}
\put(0,66){\line(1,0){80}}
\put(80,66){\line(0,-1){60}}
\dashline{3}(40,6)(40,66)
\put(10,6){\line(1,3){20}}
\put(50,6){\line(1,3){20}}
\put(50.3,6){\oval(40,30)[t]}
\put(30.3,66){\oval(40,30)[b]}
\end{picture}
$$
Multiplication is by concatenation of diagrams, so $\alpha\beta$ is obtained by putting $\alpha$ under $\beta$, interpreted as a basis vector by 
erasing closed circles in the interior of resulting diagram and multiplying 
by the scale factor
$\delta$ each time such a circle is removed.
For example, for $\alpha$ and $\beta$ as above, we have:
$$
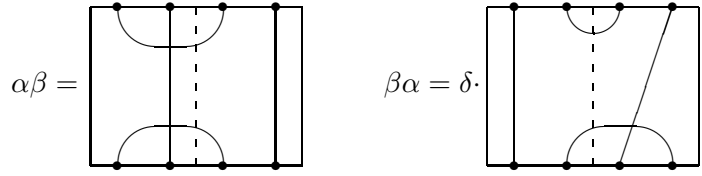
\begin{picture}(150,72)
\put(-30,34){$\alpha\beta=$}
\put(8,4.2){$\scriptstyle\bullet$}
\put(28,4.2){$\scriptstyle\bullet$}
\put(48,4.2){$\scriptstyle\bullet$}
\put(68,4.2){$\scriptstyle\bullet$}
\put(8,64.2){$\scriptstyle\bullet$}
\put(28,64.2){$\scriptstyle\bullet$}
\put(48,64.2){$\scriptstyle\bullet$}
\put(68,64.2){$\scriptstyle\bullet$}
\put(0,6){\line(1,0){80}}
\put(0,6){\line(0,1){60}}
\put(0,66){\line(1,0){80}}
\put(80,66){\line(0,-1){60}}
\dashline{3}(40,6)(40,66)
\put(70,6){\line(0,1){60}}
\put(30,6){\line(0,1){60}}
\put(30.3,6){\oval(40,30)[t]}
\put(30.3,66){\oval(40,30)[b]}
\end{picture}
\begin{picture}(80,72)
\put(-39,34){$\beta\alpha=\delta \cdot $}
\put(8,4.2){$\scriptstyle\bullet$}
\put(28,4.2){$\scriptstyle\bullet$}
\put(48,4.2){$\scriptstyle\bullet$}
\put(68,4.2){$\scriptstyle\bullet$}
\put(8,64.2){$\scriptstyle\bullet$}
\put(28,64.2){$\scriptstyle\bullet$}
\put(48,64.2){$\scriptstyle\bullet$}
\put(68,64.2){$\scriptstyle\bullet$}
\put(0,6){\line(1,0){80}}
\put(0,6){\line(0,1){60}}
\put(0,66){\line(1,0){80}}
\put(80,66){\line(0,-1){60}}
\dashline{3}(40,6)(40,66)
\put(10,6){\line(0,1){60}}
\put(50,6){\line(1,3){20}}
\put(50.3,6){\oval(40,30)[t]}
\put(40.3,66){\oval(20,20)[b]}
\end{picture}
$$
The algebra $B_{r,s}(\delta)$ is semisimple whenever
$\delta \notin \{2-r-s,3-r-s,\dots,r+s-2\}$. 
Introduce two more sets of bipartitions:
\begin{align}
\Lambda_{r,s} &:= \left\{\lambda =(\lambda^L,\lambda^R)\:\big|\:\lambda^L\vdash r-t, \lambda^R\vdash s-t,0 \leq t \leq \min(r,s)\right\},\\
\dot\Lambda_{r,s} &:= \left\{\begin{array}{ll}
\Lambda_{r,s}&\text{if $r \neq s$ or $\delta \neq 0$ or $r=s=0$,}\\
\Lambda_{r,s}\setminus\{(\varnothing,\varnothing)\}&\text{otherwise.}
\end{array}\right.\label{wed}
\end{align}
Then the isomorphism classes of irreducible $B_{r,s}(\delta)$-modules are parametrised in a canonical way by the set $\dot\Lambda_{r,s}$; we write
$D_{r,s}(\lambda)$ for the irreducible corresponding to $\lambda \in \dot \Lambda_{r,s}$.

We assume henceforth that $\delta\in \Z$
and identify bipartitions 
with certain weight diagrams, so that 
$\lambda = (\lambda^L,\lambda^R)$ corresponds to 
the weight diagram in which
the vertices $\lambda^L_1,\lambda^L_2-1,\lambda^L_3-2,\dots$ are labelled $\up$
and the vertices 
$1-\delta-\lambda^R_1, 2-\delta-\lambda^R_2,3-\delta-\lambda^R_3,\dots$
are labelled $\down$. This is a different rule from the one in section 3.
For example, if $\delta=-2$ then
\begin{align*}
(\varnothing,\varnothing)&=
\!\!\!\begin{picture}(240,12)
\put(90,14){$_0$}
\put(126,14){$_{-\delta}$}
\put(8,-.3){$\cdots$}
\put(227,-.3){$\cdots$}
\put(25,2.3){\line(1,0){198}}
\put(30,-2.4){$\up$}
\put(50,-2.4){$\up$}
\put(70,-2.4){$\up$}
\put(90,-2.4){$\up$}
\put(130,-.4){$\circ$}
\put(110,-.4){$\circ$}
\put(150,2.4){$\down$}
\put(170,2.4){$\down$}
\put(190,2.4){$\down$}
\put(210,2.4){$\down$}
\end{picture}\,,
\\ 
((2^21),(32))&=
\!\!\!\begin{picture}(240,0)
\put(8,-.3){$\cdots$}
\put(227,-.3){$\cdots$}
\put(25,2.3){\line(1,0){198}}
\put(30,-2.4){$\up$}
\put(70,-2.4){$\up$}
\put(110,-2.4){$\up$}
\put(50,-.4){$\circ$}
\put(130,.4){$\cross$}
\put(150,-.4){$\circ$}
\put(170,-.4){$\circ$}
\put(90,2.4){$\down$}
\put(190,2.4){$\down$}
\put(210,2.4){$\down$}
\end{picture}\,.
\end{align*}
Note now there are always infinitely many vertices labelled $\up$ to the left 
and infinitely many vertices labelled $\down$ to the right.
Let $\Lambda$ denote the set of all weight diagrams arising from all bipartitions in this way. Let $K(\delta)$ be the arc algebra $K_\La$ from section 4 for this choice
of $\Lambda$, and denote its irreducible modules by $L(\lambda)$ for $\lambda \in \Lambda$.

\begin{Theorem}\label{fin}
For $\delta \in \Z$,
there is a Morita equivalence between the walled Brauer algebra $B_{r,s}(\delta)$ and the finite dimensional 
 algebra $K_{r,s}(\delta) := e_{r,s} K(\delta) e_{r,s}$, where
$$
e_{r,s} := \sum_{\lambda \in \dot\Lambda_{r,s}} e_\lambda\in K(\delta).
$$
Under the equivalence, the irreducible $B_{r,s}(\delta)$-module
$D_{r,s}(\lambda)$ corresponds to the irreducible $K_{r,s}(\delta)$-module
$L_{r,s}(\lambda) := e_{r,s} L(\lambda)$, for $\lambda \in \dot \Lambda_{r,s}$.
\end{Theorem}

Assume at last that $\Lambda_{r,s} = \dot\Lambda_{r,s}$ in (\ref{wed}).
This assumption ensures that $\dot\Lambda_{r,s}$ is an ideal in the poset $(\Lambda,\geq)$, where $\geq$ is the Bruhat order from section 3.
Hence the algebra $K_{r,s}(\delta)$ is a standard Koszul algebra
with weight poset $\Lambda_{r,s}$, as follows from
Theorem~\ref{nw} (though some care is needed
since $K(\delta)$ itself does not satisfy the requirement from the opening sentence of that theorem). So Theorem~\ref{fin} implies in particular that $B_{r,s}(\delta)$ is Koszul.

\vspace{2mm}
\noindent{\em Notes.}
We refer to \cite{CD} for a detailed account of the representation theory of the walled Brauer algebra in both the semisimple and non-semisimple cases. Essentially the same rule as described in this section for converting bipartitions to weight diagrams appears already in \cite[$\S$4]{CD}, and Theorem~\ref{fin} is implicitly conjectured in
\cite[Remark 9.4]{CD}.
Theorem~\ref{fin} is proved in \cite{BSnew}, as an application of the results about $GL(m|n)$ from \cite{BS4} (see 
the previous section), together with the Schur-Weyl duality between
$GL(m|n)$ and $B_{r,s}(m-n)$ arising from their commuting actions on mixed tensor space
$V^{\otimes r} \otimes (V^*)^{\otimes s}$. 
\fi

\end{document}